\documentclass[review,final]{siamltex}

\def\X{\mathbf{X}}
\def\Y{\mathbf{Y}}
\def\Z{\mathbf{Z}}
\def\fai{\pmb{\phi}}
\def\E{\mathbf{E}}
\def\k{\mathbf{k}}

\usepackage{amsfonts,amssymb,amsmath} %%%add
\usepackage{color} %%%add
\usepackage{graphicx} %%%add
\usepackage{epstopdf,mathrsfs} %%%add
\usepackage[bf,SL,BF]{subfigure} %%%add
\usepackage{fancyhdr} %%%add
\usepackage{CJK} %%%add
\usepackage{caption} %%%add
\usepackage{wrapfig} %%%add
\usepackage{cases} %%%add
\usepackage{bm}%%%add
\usepackage{multirow}%%%add
\usepackage{url}
\newtheorem{remark}{Remark}[section] %%%add

\title{Nonlocal Diffusion Operators for Normal and Anomalous Dynamics  \thanks{This work was partially supported by NSFC 11421101, 11421110001 and 11671182.}}
\author{ Weihua Deng\thanks{School of Mathematics and Statistics, Gansu Key Laboratory of Applied Mathematics and Complex Systems, Lanzhou University, Lanzhou 730000, P.R. China. Email: dengwh@lzu.edu.cn}
  \and
   Xudong Wang\thanks{School of Mathematics and Statistics, Gansu Key Laboratory of Applied Mathematics and Complex Systems, Lanzhou University, Lanzhou 730000, P.R. China. Email: xdwang14@lzu.edu.cn}
  \and
  Pingwen Zhang\thanks{School of Mathematical Sciences, Laboratory of Mathematics and Applied Mathematics, Peking University, Beijing 100871,  P.R. China. Email: pzhang@pku.edu.cn}
}

\begin{document}

\maketitle

\begin{abstract}
The Laplacian $\Delta$ is the infinitesimal generator of isotropic Brownian motion, being the limit process of normal diffusion, while the fractional Laplacian $\Delta^{\beta/2}$ serves as the infinitesimal generator of the limit process of isotropic L\'{e}vy process. Taking limit, in some sense, means that the operators can approximate the physical process well after sufficient long time. We introduce the nonlocal operators (being effective from the starting time), which describe the general processes undergoing normal diffusion. For anomalous diffusion, we extend to the anisotropic fractional Laplacian $\Delta_m^{\beta/2}$ and the tempered one $\Delta_m^{\beta/2,\lambda}$ in $\mathbb{R}^n$. Their definitions are proved to be equivalent to an alternative one in Fourier space. Based on these new nonlocal diffusion operators, we further derive the deterministic governing equations of some interesting statistical observables of the very general jump processes with multiple internal states. Finally, we consider the associated initial and boundary value problems and prove their well-posedness of the Galerkin weak formulation in $\mathbb{R}^n$. To obtain the coercivity, we claim that the probability density function $m(\Y)$ should be nondegenerate.
\end{abstract}


\pagestyle{myheadings}
\thispagestyle{plain}

%%%%%%%%%%%%%%%%%%%%%%%%%%%%%%%%%%%%%%%%%%%%%%%%%%%%%%%%%%%%%%%%%%%%%%%%%%%%%%%%
%%%%%%%%%%%%%%%%%%%%%%%%%%%%%%%%%%%%%%%%%%%%%%%%%%%%%%%%%%%%%%%%%%%%%%%%%%%%%%%%

\section{Introduction}\label{Sec1}

Diffusion phenomena are ubiquitous in the natural world, which describe the net movements of the microscopical molecules or atoms from a region of high concentration to a region of low concentration. The speed of diffusion can be characterized by the second moment of the particle trajectories $\langle x^2(t)\rangle\thicksim t^\alpha$. It is called normal diffusion if $\alpha=1$ and anomalous diffusion \cite{Metzler:00,Sancho:04} if $\alpha\neq1$. The scaling limits of all the processes undergoing normal diffusion are Brownian motion. But without the scaling limits, most of the time, they are pure jump processes. For anomalous diffusion, the processes are always characterized by long-range correlation or broad distribution. The former includes fractional Brownian motion \cite{Meerschaert:12} and tempered fractional Brownian motion \cite{Chen:17,Meerschaert:13}, while the latter contains the processes with long tailed waiting time or jump length. In the framework of continuous time random walks (CTRWs) \cite{Klafter:11,Montroll:65}, any one of the first moment of waiting time and the second moment of jump length diverging leads to the anomalous dynamics. If we extend to the processes with multiple internal states \cite{Xu:17}, then the diffusion phenomena will depend on the distribution of each internal states, transition matrix and initial distribution, involving more complex dynamics.

There are many microscopic/stochastic models to describe normal and anomalous diffusions and many different ways of deriving the macroscopic/deterministic equations governing the probability distribution functions of some particular statistical observables of the stochastic processes. For normal diffusion, in mathematical community, most people are more familiar with the deriving procedure based on the law of mass conservation and the assumption of Fick's law.
%the law of mass conservation and Fick's law contribute to the %classical heat equation.
The commonly used stochastic models include CTRWs, Langevin type equations, and L\'{e}vy processes. The CTRWs consist of two important random variables, i.e., the waiting time $\xi$ and jump length $\eta$. If both the first moment of waiting time $\langle \xi\rangle$ and the second moment of jump length $\langle \eta^2\rangle$ are finite, after taking the scaling limit, the CTRWs converge to Brownian motion.
On the contrary, if $\langle \xi\rangle$ diverges and  $\langle \eta^2\rangle$  is finite, the CTRW describes subdiffusion, while it characterizes superdiffusion if $\langle \xi\rangle$ is bounded and   $\langle \eta^2\rangle$ infinite; if both  $\langle \xi\rangle$ and $\langle \eta^2\rangle$ are unbounded, the type of diffusions is possible to be subdiffusion, superdiffusion, or even normal diffusion, depending on the dominant role played by $\xi$ or $\eta$ or that $\xi$ and $\eta$ are balanced each other.
%and $\langle \eta^2\rangle$ diverges, the CTRWs describe anomalous diffusion,
%
%a divergent $\langle \xi\rangle$ leads to the subdiffusion, while a divergent $\langle \eta^2\rangle$ leads to superdiffusion.
Two of the most important CTRW models undergoing anomalous diffusion are L\'{e}vy flights and L\'{e}vy walks. For L\'{e}vy flights, the $\xi$ with finite first moment and $\eta$ with infinite second moment are independent, and the divergence of the second moment of $\eta$ makes the processes propagate with infinite speed. Therefore, the physical realizations of such processes are quite hard and then rare.
%it has infinite propagation speed. Therefore, physical realizations of such processes are quite rare.
L\'{e}vy walks \cite{Zaburdaev:15} can remedy the divergence of the second moment of jump length $\langle \eta^2\rangle$ by coupling the distribution of $\xi$ and $\eta$. This gives rise to a class of space-time coupled processes. Different from L\'{e}vy walks, another idea to bound the second moment is to truncate the long tailed probability distribution of jump length \cite{Meerschaert:12,Mantegna:94}, i.e., modify $|\X|^{-n-\beta}$ as $e^{-\lambda|\X|}|\X|^{-n-\beta}$ with $\lambda$ being a small positive constant, leading to the tempered L\'{e}vy flights, which have the advantage of still being an infinitely divisible L\'{e}vy process. The Langevin type equations are built based on Newton's second law with noise as random forces, and the CTRW models also have their corresponding Langevin pictures \cite{Fogedby:94}. Sometimes, it is convenient to use this type of models if the external potentials are considered.

Another way to describe anomalous diffusion is the L\'{e}vy process (subordinated L\'{e}vy process, and inverse subordinated L\'{e}vy process). It is defined by its characteristic function and more convenient to deal with the stochastic process in high dimensional space. According to the L\'{e}vy-Khintchine formula \cite{Applebaum:09}, the characteristic function of L\'{e}vy process has a specific form
\begin{equation}\label{LevyCharfunction}
  \E(e^{i\k\cdot\X})=\int_{\mathbb{R}^n}e^{i\k\cdot\X}p(\X,t)d\X=e^{t \Phi(\k)},
\end{equation}
where
\begin{equation}\label{LKformula}
  \Phi(\k)=i\k\cdot\mathbf{b}-\frac{1}{2}\k\cdot\mathbf{a}\k + \int_{\mathbb{R}^n\backslash\{0\}} \Big[e^{i\k\cdot\Y}-1-i\k\cdot\Y_{\chi_{\{|\Y|<1\}}}\Big]\nu(d\Y),
\end{equation}
with $\mathbf{b}\in\mathbb{R}^n$, and $\mathbf{a}$ is a positive definite symmetric $n\times n$ matrix, $\chi_I$ is the indicator function of the set $I$, $\nu$ is a finite L\'{e}vy measure on $\mathbb{R}^n\backslash\{0\}$, implying that $\int_{\mathbb{R}^n\backslash\{0\}}\min\{1,|\Y|^2\}\nu(d\Y)<\infty$.
If we take $\mathbf{a}$ and $\mathbf{b}$ to be zero and $\nu$ to be a rotationally symmetric (tempered) $\beta$-stable L\'{e}vy measure
\begin{equation}\label{rot_sym}
  \nu(d\Y)= c_{n,\beta} |\Y|^{-n-\beta} d\Y \quad \textrm{or} \quad  \nu(d\Y)= c_{n,\beta,\lambda} e^{-\lambda|\Y|}|\Y|^{-n-\beta} d\Y,
\end{equation}
then its probability density function (PDF) of the position of the particles solves
\begin{equation}
  \frac{\partial p(\X,t)}{\partial t} = \Delta^{\beta/2} \,p(\X,t) \quad \textrm{or} \quad \frac{\partial p(\X,t)}{\partial t} = \Delta^{\beta/2,\lambda}\, p(\X,t),
\end{equation}
where the operators $\Delta^{\beta/2}$ and $\Delta^{\beta/2,\lambda}$ are defined in \cite[Eq. 34]{Deng:17} by Fourier transform $\hat{g}(\k):=\mathscr{F}[g(\X)](\k)=\int_{\mathbb{R}^n}e^{i\k\cdot\X}g(\X)d\X$ with
\begin{equation}\label{Lap-tempLap}
\begin{aligned}
  & \mathscr{F}[\Delta^{\beta/2}g(\X)]=-|\k|^\beta\hat{g}(\k) \textrm{ and }
  \\
&
  \mathscr{F}[\Delta^{\beta/2,\lambda}g(\X)]= (-1)^{\lceil\beta\rceil} \Big((\lambda^2+|\k|^2)^{\beta/2}-\lambda^\beta+\mathcal{O}(|\k|^2)\Big)\hat{g}(\k);
\end{aligned}
\end{equation}
here $\beta\in(0,1)\cup(1,2)$, and ${\lceil\beta\rceil}$ denotes the smallest integer that is bigger than or equal to $\beta$.
A similar operator $(\lambda^{1/\beta}+|\k|^2)^\beta-\lambda$ appears in \cite[Eq. 3]{Ryznar:02}, where the only difference is the term $\mathcal{O}(|\k|^2)$. However, their physical background is completely different. The term $\mathcal{O}(|\k|^2)$ in \eqref{Lap-tempLap} is strictly derived in \cite[Eq. 34]{Deng:17}, where we consider the compound Poisson processes with tempered power law jump lengths, i.e., take the L\'{e}vy measure $\nu(d\Y)$ to be $e^{-\lambda|\Y|}|\Y|^{-n-\beta}$. But for the formula in \cite[Eq. 3]{Ryznar:02}, it is inspired by the Schr\"{o}dinger operator with the free Hamiltonian of the form $H_0 = (\lambda^2-\Delta)^{1/2}-\lambda$ in \cite{Carmona:90}, and naturally extended to the form $(\lambda^{1/\beta}+|\k|^2)^\beta-\lambda$ with fractional order $\beta$.

The two equations in \eqref{Lap-tempLap} describe the isotropic movements of microscopic particles with (tempered) L\'{e}vy distribution. At the same time, in the natural world, anisotropic motions are very popular. So we need to develop  models for characterizing the corresponding physical reality.
%The models for describing anisotropic motions achieve better approximation to the physical reality.
Compte \cite{Compte:97} generalized the scheme of CTRWs and showed the diffusion-advection equation and the mean square displacement in three kinds of shear flows. Meerschaert et al \cite{Meerschaert:99} made an extension to higher dimensions and provided an operator being mixture of directional derivatives taken in each radial direction, where the operator was directly given in Fourier space and the associated fractional advection-dispersion equation was derived. Ervin and Roop \cite{Ervin:07} discussed directional integral and directional differential operators in two dimensions, and defined the appropriate fractional directional derivative spaces. For more details we refer the interested readers to these literatures and the references cited therein. In this paper, we start from the compound Poisson process to discuss more general nonlocal normal diffusion and anomalous diffusion. It is well known that, most of the time,  anomalous diffusions are described by nonlocal differential equations.
%People know that anomalous diffusion equations are mostly derived from nonlocal processes.
But for normal diffusion, a compound Poisson process with Gaussian jumps indeed leads to a nonlocal differential equations. For the isotropic movement and the movement just allowed in axis directions, their associated diffusion equations are different, though the scaling limit makes them become the same classical diffusion equation.
%of nonlocal diffusion both contribute to classical diffusion equations characterizing normal diffusion.
%
%
%With anisotropic movements, their associated diffusion equations are different, though the scaling limits of nonlocal diffusion both contribute to classical diffusion equations characterizing normal diffusion.
For the nonlocal normal diffusion, we still can discuss the problem of escape probability \cite{Chechkin:03,DengWW:17,Koren:07} and the way of specifying the boundary conditions of their corresponding macroscopic equations is the same as the models for anomalous diffusion. We also discuss the anomalous diffusion undergoing anisotropic movements in $\mathbb{R}^n$, and derive the associated diffusion equations with anisotropic tempered fractional Laplacian $\Delta_m^{\beta/2,\lambda}$ ($\lambda=0$ corresponds to the one without tempering and the subscript $m$ means that this new operator depends on the probability density function $m(\theta)$ or $m(\Y)$ first appeared in \eqref{def_m}).
%For anomalous diffusion, we also assume that particles undergoing the anisotropic movements in $\mathbb{R}^n$, then derive the associated diffusion equations with anisotropic tempered fractional Laplacian $\Delta_m^{\beta/2,\lambda}$ ($\lambda=0$ corresponds to the one without tempering and the subscript $m$ means this new operator depends on the probability density function $m(\theta)$ or $m(\Y)$ first appeared in \eqref{def_m}).
Similar to the operator in \cite[Eq. 2]{Meerschaert:99}, we also give the tempered one in Fourier space and show its equivalence with the just derived one $\Delta_m^{\beta/2,\lambda}$. Then we discuss the space fractional partial differential equations (PDEs) with the anisotropic tempered fractional Laplacian $\Delta_m^{\beta/2,\lambda}$ in $\mathbb{R}^n$, endowed with generalized Dirichlet and Neumann boundary conditions, and prove  their well-posedness. One of the key requests is to have the coercivity of the variational formulation of the PDEs in $\mathbb{R}^n$, being proved by the technique in $\mathbb{R}^1$ presented in \cite{Zhang:17} under some assumptions on the probability density function $m(\Y)$.

% we propose the conditions that the probability density function $m(\Y)$ should satisfy and use the technique in $\mathbb{R}^1$ presented in \cite{Zhang:17}.

All the models mentioned above are for the diffusion with single internal state, implying that the processes have the same distributions of waiting time $\xi$ and jump length $\eta$ throughout the time. Intrigued by applications, e.g., the particles moving in multiphase viscous liquid composed of materials with different chemical properties, we further generalize the processes with multiple internal states.
%Besides, the models above are all characterised by single internal state, which means the processes keep the same distribution of waiting time $\xi$ and jump length $\eta$ all the time. We may further generalize to the processes with multiple internal states. The fractional Poisson processes with multiple internal states have a lot of potential applications, e.g., the particles moving in multiphase viscous liquid composed of materials with different chemical properties.
In fact, the case of two internal states is considered in \cite{Godec:17,Pollak:93} with applications, including trapping in amorphous semiconductors, electronic burst noise, movement in systems with fractal boundaries, the digital generation of $1/f$ noise, and ionic currents in cell membranes; Niemann et al \cite{Niemann:16} detailedly investigate a stochastic signal with multiple states, in which each state has an associated joint distribution for the signal's intensity and its holding time. Xu and Deng \cite{Xu:17} derived the Fokker-Planck and Feymann-Kac \cite{Turgeman:09,Wu:16} equations for the particles undergoing the anomalous diffusion with multiple temporal internal states. Here, we further present the fractional Fokker-Planck and Feymann-Kac equations with multiple internal states, both temporally and spatially.

The rest of this paper is organized as follows. In Section 2, we show two kinds of processes with Gaussian jumps, leading to different nonlocal macroscopic equations describing normal diffusions. More general anisotropic processes undergoing anomalous diffusions are discussed in Section 3, and we also give two kinds of definitions of anisotropic (tempered) fractional Laplacian for two different motivations and prove their equivalences.  In Section 4, the fractional Fokker-Planck and Feymann-Kac equations of anisotropic (tempered) fractional Laplacian with multiple internal states are derived.  The initial and boundary value problems with generalized Dirichlet and Neumann boundary conditions are given in Section 5, and their well-posednesses are proved in Section 6. We conclude the paper with some discussions in the last section.

\section{Nonlocal normal diffusion}\label{Sec2}

As all we know, except Brownian motion with drift, the paths of all other proper L\'{e}vy processes are discontinuous. From the viewpoint of \cite{Dybiec:06,Deng:17}, the PDEs governing the PDFs of these processes should be endowed with the generalized boundary conditions, since the  boundary $\partial\Omega$ itself can not be hit by the majority of discontinuous sample trajectories. For nonlocal normal diffusion, it is a pure jump process with Gaussian jumps.
Therefore, the boundary conditions of their corresponding PDEs should be specified on the domain $\mathbb{R}^n\backslash\Omega$. By the central limit theorem, the scaling limits of all these processes are Brownian motion. But without scaling limit, these processes are different and should be distinguished.

Now we consider the compound Poisson process with Gaussian jump length, in which Poisson process is taken as the renewal process. We denote  Poisson process by $N(t)$ satisfying $P\{N(t)=n\}=\frac{(\zeta t)^n}{n!}e^{-\zeta t}$, where the rate $\zeta>0$ denotes the mean number of jumps per unit time. Then the compound Poisson process is defined as $\X(t)=\sum_{j=0}^{N(t)}\X_j$, where $\X_j$ are i.i.d. random variables and their length obeys  Gaussian distribution. The characteristic function of $\X(t)$ has a specific form as \cite[Eq. 9]{Deng:17}
\begin{equation}\label{char_X}
  \E(e^{i\k\cdot\X})=\int_{\mathbb{R}^n}e^{i\k\cdot\X}p(\X,t)d\X=e^{\zeta t (\Phi_0(\k)-1)},
\end{equation}
where $\Phi_0(\k)=\E(e^{i\k\cdot\X_j}), j=0,1,\cdots$,N(t). Denoting the probability measure of the jump length $\X_j$ by $\nu(d\Y)$, we have
\begin{equation}\label{LK}
  \Phi_0(\k)-1=\int_{\mathbb{R}^n} (e^{i\k\cdot\Y}-1) \nu(d\Y),
\end{equation}
which is the same as the L\'{e}vy-Khintchine formula \eqref{LKformula} by taking $\mathbf{a}=0$ and $\mathbf{b}'=0$ ($\mathbf{b}'$ contains $\mathbf{b}$ and the third term in the integral of \eqref{LKformula}).
%where $\mathbf{a}$ and $\mathbf{b}$ are assumed to be zeros for modelling a pure jump processes considered here. While $\nu(d\Y)=0$ models a classical diffusion.
%Moreover, the term $i\k\cdot\Y$ are artificially added. It is only valid for high singularity in $\nu(d\Y)$ (e.g. $\beta$-stable L\'{e}vy distribution with $1<\beta<2$), and should be omitted for weak singularity in $\nu(d\Y)$ (e.g. $\beta$-stable L\'{e}vy distribution with $0<\beta<1$) or nonsingularity (e.g. Gaussian distribution here). Otherwise, this term $i\k\cdot\Y$ makes no sense in compound Poisson process.
Although the length of $\X_j$ obeys Gaussian distribution, the distribution of the direction of the movement has many different  choices. Here, we consider two specific cases in two dimensional space, and derive their corresponding deterministic equations governing the PDF of position of the particles undergoing normal diffusion. The first case is that the particles spread uniformly in all directions while the second one is that the particles move only in horizontal and vertical direction.
Considering the definition of Fourier transform and \eqref{char_X}, we have
\begin{equation}
  \hat{p}(\k,t)=e^{\zeta t (\Phi_0(\k)-1)},
\end{equation}
which implies that the equation in $\k$ space is
\begin{equation}\label{kspace}
  \frac{\partial\hat{p}(\k,t)}{\partial t}=\zeta(\Phi_0(\k)-1)\hat{p}(\k,t).
\end{equation}
Next, we give the specific expressions of $\Phi_0(\k)$ (or $\nu(d\Y)$) for these two cases.

{\bf Case 1}: Since the particles spread uniformly in all directions, $\nu(d\Y)$ is taken as
\begin{equation*}
  \nu(d\Y)=\frac{1}{2\pi\sigma^2}e^{-\frac{|\Y|^2}{2\sigma^2}}d\Y,
\end{equation*}
where $\sigma^2$ is the variance.
Then we obtain
\begin{equation}\label{char_case1}
  \Phi_0(\k)-1=e^{-\frac{1}{2}\sigma^2|\k|^2}-1,
\end{equation}
which implies
\begin{equation}\label{macro_case1}
  \frac{\partial p(\X,t)}{\partial t}=-\frac{\zeta}{2\pi\sigma^2}\int_{\mathbb{R}^2}e^{-\frac{|\X-\Y|^2}{2\sigma^2}}(p(\X,t)-p(\Y,t))d\Y,
\end{equation}
by taking the inverse Fourier transform
\begin{equation*}
  \begin{split}
    \mathscr{F}^{-1}[(\Phi_0(\k)-1)\hat{p}(\k,t)] &=\mathscr{F}^{-1}[\Phi_0(\k)\hat{p}(\k,t)]-\mathscr{F}^{-1}[\hat{p}(\k,t)] \\
    &=\frac{1}{2\pi\sigma^2}\int_{\mathbb{R}^2}e^{-\frac{|\X-\Y|^2}{2\sigma^2}}p(\Y,t))d\Y-p(\X,t) \\
    &=-\frac{1}{2\pi\sigma^2}\int_{\mathbb{R}^2}e^{-\frac{|\X-\Y|^2}{2\sigma^2}}(p(\X,t)-p(\Y,t))d\Y.
  \end{split}
\end{equation*}

{\bf Case 2}: Since the particles spread in either horizontal or vertical direction, we take $\nu(d\Y)$ to be
\begin{equation*}
  \nu(d\Y)=\frac{1}{2(2\pi\sigma^2)^{\frac{1}{2}}}e^{-\frac{|y_1|^2}{2\sigma^2}}\delta(y_2)d\Y
          +\frac{1}{2(2\pi\sigma^2)^{\frac{1}{2}}}e^{-\frac{|y_2|^2}{2\sigma^2}}\delta(y_1)d\Y.
\end{equation*}
Similar to Case 1, we have
\begin{equation}\label{char_case2}
  \Phi_0(\k)-1=\frac{1}{2}e^{-\frac{1}{2}\sigma^2|k_1|^2}
            +\frac{1}{2}e^{-\frac{1}{2}\sigma^2|k_2|^2}-1,
\end{equation}
and thus derive the equation
\begin{equation}\label{macro_case2}
\begin{split}
  \frac{\partial p(\X,t)}{\partial t}=-\frac{\zeta}{2(2\pi\sigma^2)^{\frac{1}{2}}}
  \Big(&\int_{\mathbb{R}}e^{-\frac{|x_1-y_1|^2}{2\sigma^2}}(p(x_1,x_2,t)-p(y_1,x_2,t))dy_1\\
  &+\int_{\mathbb{R}}e^{-\frac{|x_2-y_2|^2}{2\sigma^2}}(p(x_1,x_2,t)-p(x_1,y_2,t))dy_2\Big).
\end{split}
\end{equation}

From \eqref{macro_case1} and \eqref{macro_case2}, it can be noted that different ways of movement of microscopic particles lead to different macroscopic equations. Furthermore, these macroscopic equations are both nonlocal, and should be endowed with the generalized boundary conditions. But the scaling limits of the Gaussian jump processes of the above two cases are both Brownian motion. In fact, let $\frac{1}{\zeta}\rightarrow0$ and $\sigma^2\rightarrow0$, while the product
\begin{equation*}
  \lim_{1/\zeta\rightarrow0,\sigma^2\rightarrow0} \frac{1}{2}\zeta\sigma^2=K_1
\end{equation*}
is kept finite, where $K_1$ is the diffusion coefficient with unit $[\textrm{cm}^2]/[\textrm{s}]$ \cite{Barkai:00}. Then, both \eqref{char_case1} and \eqref{char_case2} become, up to a multiplier,
\begin{equation*}
  \Phi_0(\k)-1=-\frac{1}{2}\sigma^2|\k|^2,
\end{equation*}
resulting in the classical heat equation
\begin{equation}\label{ClassicalEq}
  \frac{\partial p(\X,t)}{\partial t}=K_1\Delta p(\X,t),
\end{equation}
where $\Delta$ is the usual Laplacian in $\mathbb{R}^2$.
\begin{figure}[ht]
\begin{minipage}{0.4\linewidth}
  \centerline{\includegraphics[scale=0.4]{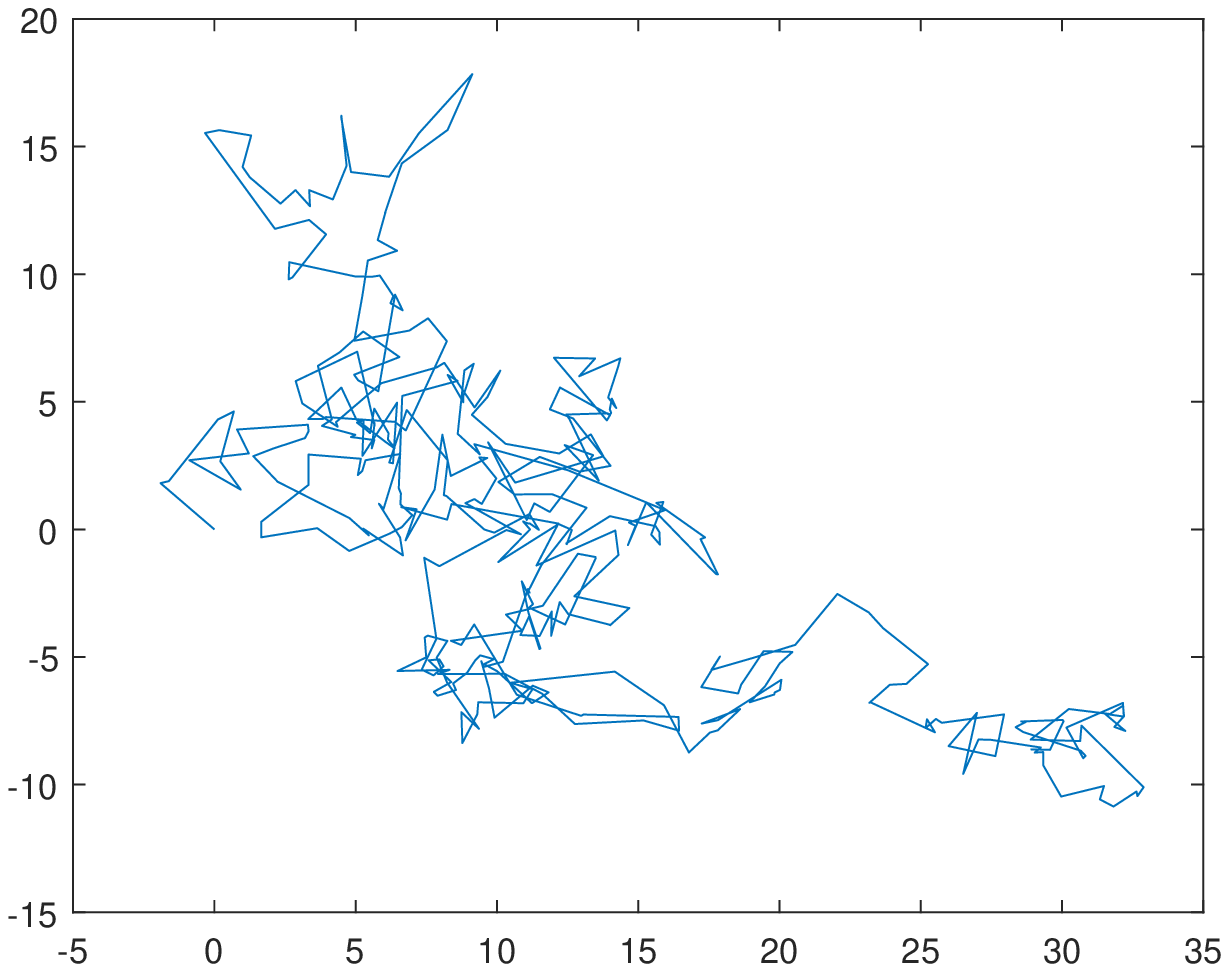}}
  \centerline{(a): Case 1 (400 steps)}
\end{minipage}
\hfill
\begin{minipage}{0.4\linewidth}
  \centerline{\includegraphics[scale=0.4]{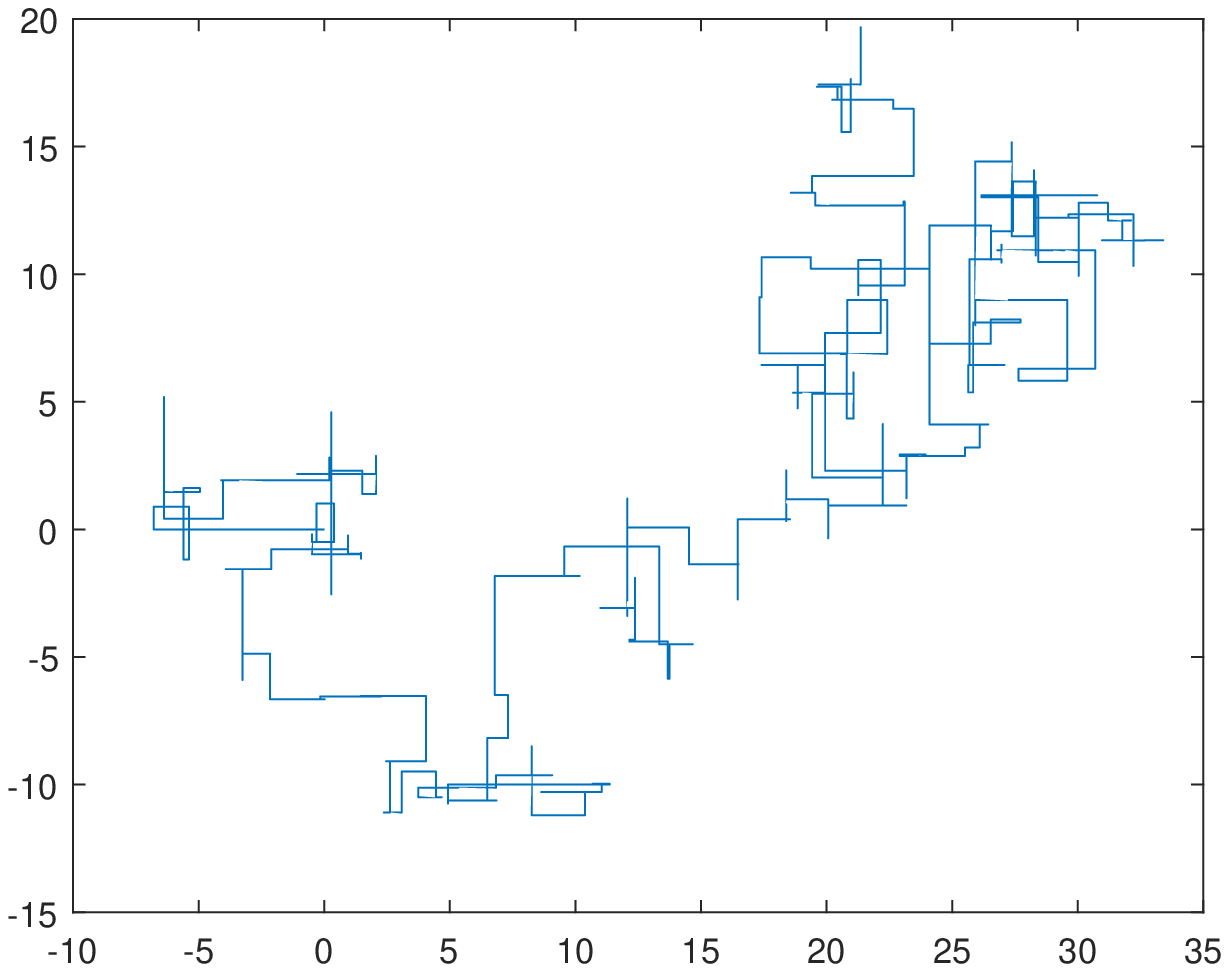}}
  \centerline{(b): Case 2 (400 steps)}
\end{minipage}
\vfill
\begin{minipage}{0.4\linewidth}
  \centerline{\includegraphics[scale=0.4]{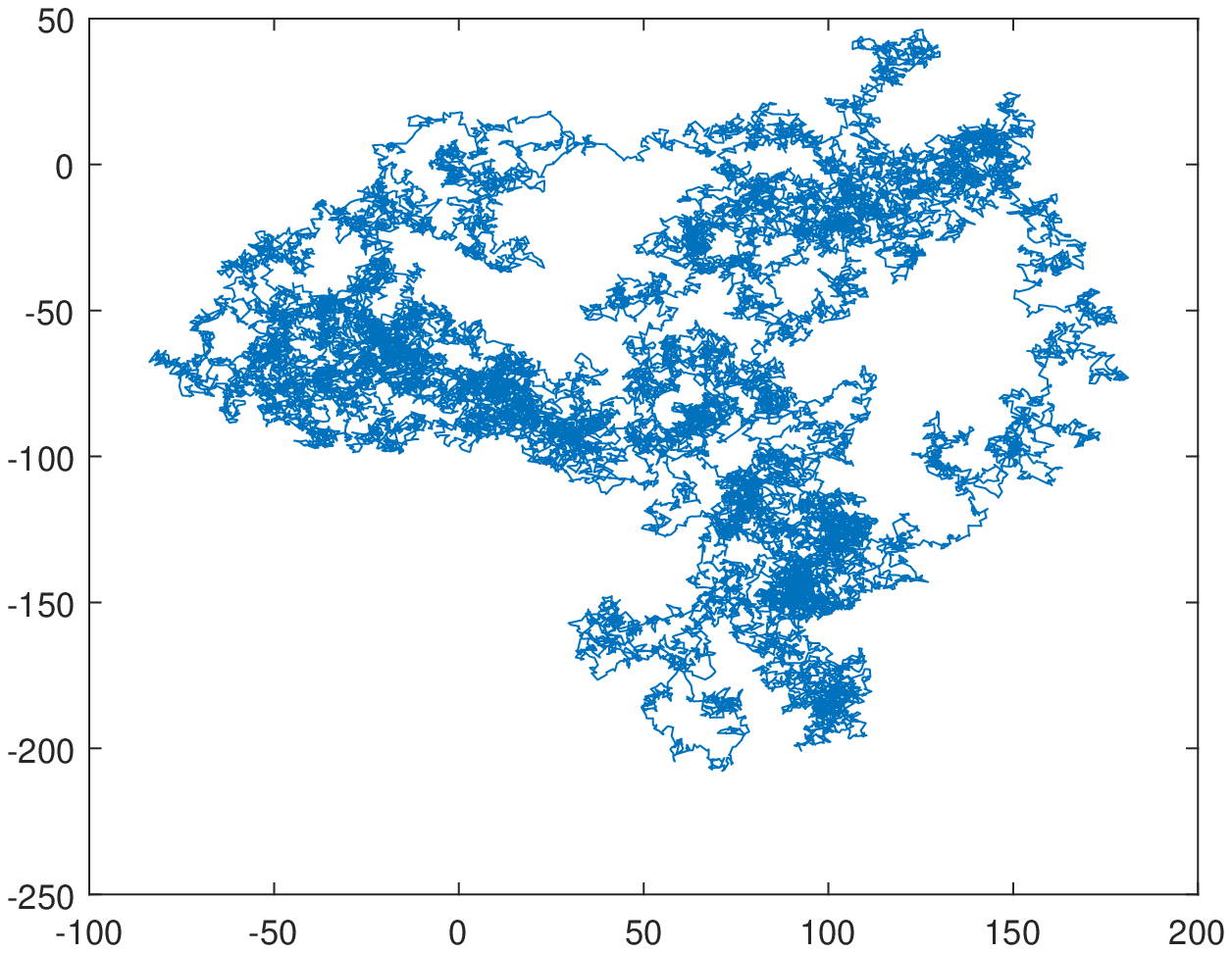}}
  \centerline{(c): Case 1 (40000 steps)}
\end{minipage}
\hfill
\begin{minipage}{0.4\linewidth}
  \centerline{\includegraphics[scale=0.4]{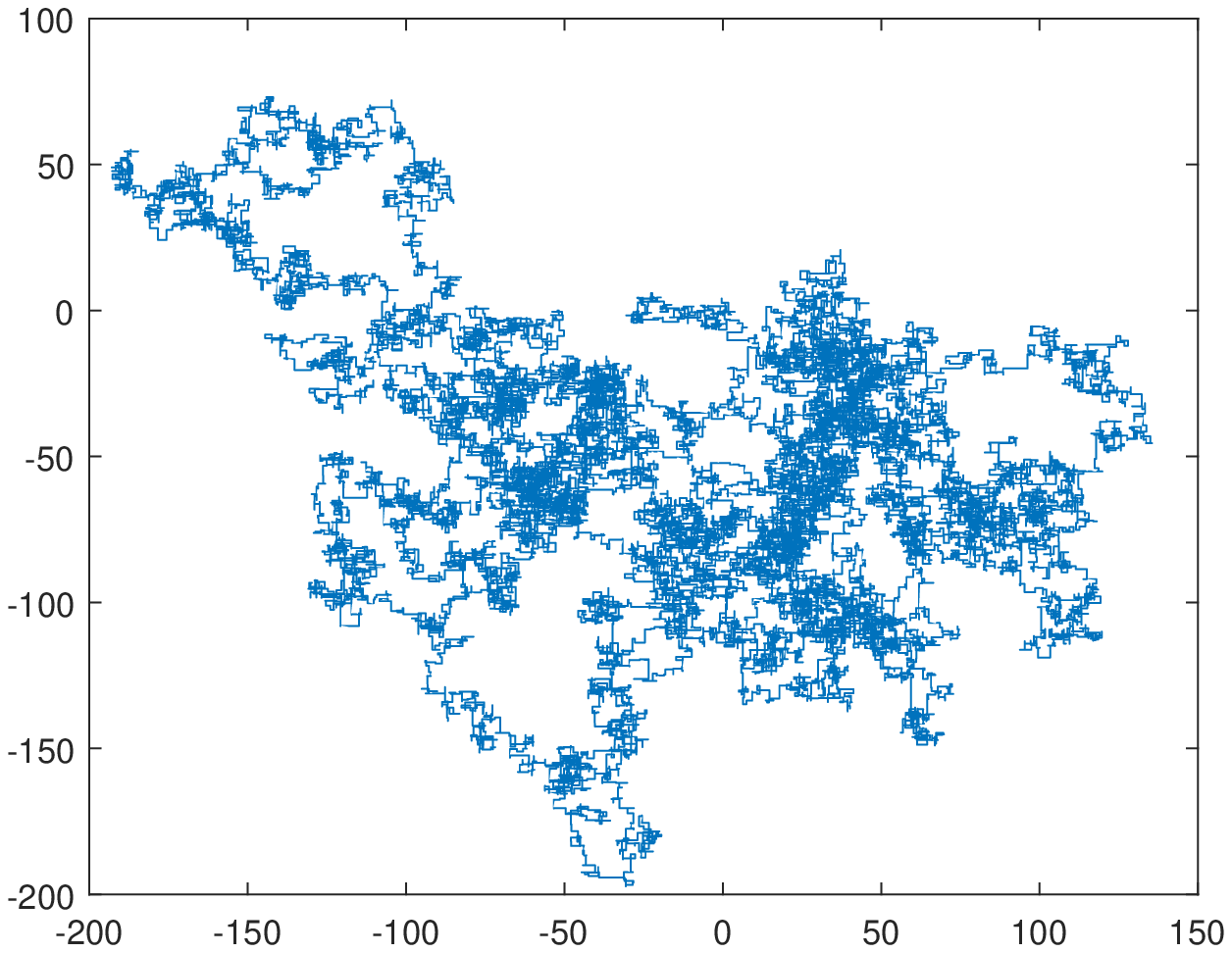}}
  \centerline{(d): Case 2 (40000 steps)}
\end{minipage}
  \caption{Random trajectories of Gaussian jumps in Case 1 and Case 2 with 400 steps in the top row and 40000 steps in the bottom row.}\label{fig0}
\end{figure}
To illustrate the relationship between Case 1 and Case 2, we simulate the trajectories of the particles undergoing Gaussian jumps. Two pictures in the top row are for the 400 jumps performed uniformly   (a) and just in horizontal-vertical direction (b), while another two pictures in the bottom row display 40000 jumps, respectively. The differences between Case 1 and Case 2 are apparent for a relatively small number of jumps. But after many thousands of jumps, they gradually disappear, as both processes are converging to the same Brownian motion.

Besides the two cases above, more generally, the particles can move in a variety of different ways, depending on the environment. There may be more particles spreading in one direction or some particles spreading faster in another direction. This phenomenon is named as anisotropic diffusion, and can be expressed clearly by the L\'{e}vy measure $\nu(d\Y)$. More precisely, still in two dimensional space, by polar coordinates transformation, take $\nu(d\Y)$ to be
\begin{equation}\label{def_m}
  \nu(d\Y)=c_m\exp\left[-\frac{r^2}{2\sigma^2_\theta}\right] m(\theta) r dr d\theta,
\end{equation}
where $c_m>0$ is the normalized parameter, $r\geq0$, $\theta\in[0,2\pi)$ denotes the different directions, $m(\theta)$ denotes the probability distribution of particles spreading in $\theta$-direction, satisfying $m(\theta)\geq0$, $\int_{0}^{2\pi}m(\theta)d\theta=1$, and $\sigma_\theta$ denotes the possibly different variance or speed of particles spreading in $\theta$-direction. Different from \eqref{rot_sym}, this $\nu(d\Y)$ contains a new probability density function $m(\theta)$ which only depends on direction. Turning back to the Cartesian coordinate system and following \eqref{LK}, we have
\begin{equation*}
  \Phi_0(\k)-1= c_m \int_{\mathbb{R}^2}
  (e^{i\k\cdot\Y}-1)\exp\left[-\frac{|\Y|^2}{2\sigma^2_{\Y}}\right] m(\Y)d\Y,
\end{equation*}
where the probability density function $m(\theta)$ is abused by $m(\Y)$ and $\Y\in \mathbb{R}^n\backslash\{0\}$ is in the Cartesian coordinate system, while it really means $m\left(\frac{\Y}{|\Y|}\right)$, only depending on the radial direction of $\Y$. The notation $m(\Y)$ will be used in the subsequent sections.
Then similarly to \eqref{macro_case1} and \eqref{macro_case2}, we can derive the equation
\begin{equation}\label{aniso_Gaussian}
  \frac{\partial p(\X,t)}{\partial t}=-\zeta c_m \int_{\mathbb{R}^2} (p(\X,t)-p(\Y,t))
  \exp\left[-\frac{|\X-\Y|^2}{2\sigma^2_{\X-\Y}}\right]m(\X-\Y)d\Y.
\end{equation}
If we take $\sigma_{\theta}=\sigma$, $m(\theta)=(2\pi)^{-1}$ being a constant or
$m(\theta)=\frac{1}{4}(\delta(\theta)+\delta(\theta-\frac{\pi}{2})+\delta(\theta-{\pi})+\delta(\theta-\frac{3\pi}{2}))$ in \eqref{def_m},
then Eq. \eqref{aniso_Gaussian} reduces to \eqref{macro_case1} and \eqref{macro_case2}, respectively.

All the above discussions, including Case 1 and Case 2, and even the case of \eqref{def_m}, are for pure jump processes (without the scaling limit). The models are different and their associated macroscopic equations should be endowed with generalized boundary conditions. But after the scaling limit, Case 1 and Case 2 are equivalent, and only local boundary conditions for their macroscopic equations \eqref{ClassicalEq} are needed.

\section{Anisotropic anomalous diffusion}\label{Sec3}

Here, we discuss the anomalous diffusion with the property of anisotropy. Still based on the compound Poisson processes in the previous section, but with the diffusion processes being anisotropic (tempered) $\beta-$stable, we try to derive their corresponding deterministic equations undergoing anomalous diffusion. Taking $\zeta=1$ in \eqref{kspace} leads to
\begin{equation}\label{Levyf_kspace}
  \frac{\partial\hat{p}(\k,t)}{\partial t}=(\Phi_0(\k)-1)\hat{p}(\k,t),
\end{equation}
where
\begin{equation}\label{Levyf_char}
  \Phi_0(\k)-1=\int_{\mathbb{R}^n\backslash\{0\}} \Big[e^{i\k\cdot\Y}-1-i\k\cdot\Y_{\chi_{\{|\Y|<1\}}}\Big]\nu(d\Y).
\end{equation}
Here, different from \eqref{LK}, we add a term $i\k\cdot\Y$ to overcome the possible divergence of the integral of \eqref{Levyf_char} because of the possible strong singularity of $\nu(d\Y)$ at zero for the case of anomalous diffusion.
%Here, $\mathbf{a}$ and $\mathbf{b}$ are still assumed to be zeros as in \eqref{LK} for modeling a pure jump process.
For an isotropic $\beta$-stable anomalous diffusion process in $n$ dimensional space, its distribution of jump length is $c_\beta r^{-n-\beta}$, which means that
\begin{equation}\label{nu_iso}
  \nu(d\Y)=c_\beta |\Y|^{-n-\beta} d\Y.
\end{equation}
When $0<\beta<1$, the term $i\k\cdot\Y$ can be omitted due to weak singularity (the integral in \eqref{Levyf_char} is convergent at origin). If $1\leq\beta<2$, though the singularity is strong, this term can also  be omitted due to the possible symmetry of the L\'{e}vy measure $\nu(d\Y)$, i.e., $\nu(d\Y)=\nu(-d\Y)$ (the integral in \eqref{Levyf_char} at origin can be understood in the sense of Cauchy principal value). Therefore, if $1\leq\beta<2$ meets with the asymmetry of $\nu(d\Y)$, this term is required.
Based on the analyses above, we will keep the term $i\k\cdot\Y$ formally for $0<\beta<2$ in the following, though it vanishes in some appropriate situations.

Two special cases have been considered in \cite{Deng:17}, i.e., the isotropic one \eqref{nu_iso} and the horizontal-vertical one
\begin{equation}\label{nu_HV}
  \begin{split}
    \nu(d\Y)
=&\,
c_{\beta_1}|y_1|^{-1-\beta_1} \delta(y_2) \delta(y_3) \cdots \delta(y_n) d\Y  \\
&+
c_{\beta_2}|y_2|^{-1-\beta_2} \delta(y_1) \delta(y_3) \cdots \delta(y_n) d\Y  + \cdots\\
&+
c_{\beta_n}|y_n|^{-1-\beta_n} \delta(y_1) \delta(y_2) \cdots \delta(y_{n-1}) d\Y,
  \end{split}
\end{equation}
where $\beta_i\in(0,2)$ and $y_i$ is the component of $\Y$, i.e., $\Y=[y_1,y_2,\cdots,y_n]^T$.
Their corresponding macroscopic equations are
\begin{equation}\label{FracLap1}
\begin{split}
  \frac{\partial p(\X,t)}{\partial t}
  &=\Delta^{\beta/2}p(\X,t) \\
  &=-c_{n,\beta}\, \mathrm{P.V.} \int_{\mathbb{R}^n} \frac{p(\X,t)-p(\Y,t)}{|\X-\Y|^{n+\beta}}d\Y \\
\end{split}
\end{equation}
and
\begin{equation}\label{FracHV1}
  \frac{\partial p(\X,t)}{\partial t} =
  (\Delta_{x_1}^{\beta_1/2}+\Delta_{x_2}^{\beta_2/2}+\cdots+\Delta_{x_n}^{\beta_n/2} )p(\X,t),
\end{equation}
%where $\Delta_{x_i}^{\beta_i/2}$ is the same fractional Laplacian as the one in \eqref{FracLap1}, but in $\mathbb{R}^1$ w.r.t. $x_i$.
where $\Delta_{x_i}^{\beta_i/2}$ is the fractional Laplacian in $\mathbb{R}^1$ w.r.t. $x_i$.
Besides the two cases, there are also a large number of irregular motions the microscopic particles perform. In general, we call it anisotropy. With the aid of L\'{e}vy-Khintchine formula \eqref{LKformula}, we will give the concrete form of $\nu(d\Y)$ in two and three dimensions.

Following \eqref{Levyf_kspace} and \eqref{Levyf_char}, with inverse Fourier transform, we have
\begin{equation}\label{aniso1}
  \frac{\partial p(\X,t)}{\partial t} =
  \int_{\mathbb{R}^n\backslash\{0\}} [p(\X-\Y)-p(\X)+(\Y\cdot\nabla_\X p(\X))_{\chi_{[|\Y|<1]}}] \nu(d\Y),
\end{equation}
where $\nabla_\X=[\partial_{x_1},\partial_{x_2},\cdots,\partial_{x_n}]^T$.
Taking
\begin{equation}\label{aniso2}
  \nu(d\Y)= \frac{1}{|\Gamma(-\beta)|}\, \frac{m(\Y)}{|\Y|^{n+\beta}} d\Y,
\end{equation}
then \eqref{aniso1} becomes
\begin{equation}\label{anisoLap}
\begin{split}
  \frac{\partial p(\X,t)}{\partial t} = & \frac{1}{|\Gamma(-\beta)|}\,
  \int_{\mathbb{R}^n\backslash\{0\}} \left[p(\X-\Y)-p(\X)+(\Y\cdot\nabla_\X p(\X))_{\chi_{[|\Y|<1]}}\right]
  \\
  & \cdot \frac{m(\Y)}{|\Y|^{n+\beta}} d\Y.
\end{split}
\end{equation}
We can make the meaning of $m(\Y)$ clear by transforming this equation into polar coordinate system. In the two and three dimensional cases, \eqref{anisoLap} becomes, respectively,
{\small
\begin{equation*}
\begin{split}
    \frac{\partial p(\X,t)}{\partial t} =  \frac{1}{|\Gamma(-\beta)|}\,
    \int_0^\infty \int_0^{2\pi} \Bigg[& p(x_1-r\cos(\theta),x_2-r\sin(\theta))-p(x_1,x_2) \\
                  +& \left(r\cos(\theta)\frac{\partial p}{\partial x_1}+r\sin(\theta)\frac{\partial p}{\partial x_2}\right)_{\chi_{[r<1]}}\Bigg]
                  \frac{m(\theta)}{r^{1+\beta}} d\theta dr
\end{split}
\end{equation*}}
and
{\small
\begin{equation*}
\begin{split}
    &\frac{\partial p(\X,t)}{\partial t}
    \\
    &=\frac{1}{|\Gamma(-\beta)|}\,
    \int_0^\infty \int_0^\pi \int_0^{2\pi} \Bigg[ p(x_1-r\sin(\theta)\cos(\phi),x_2-r\sin(\theta)\sin(\phi),x_3-r\cos(\theta))   \\
    & ~~~~ -p(x_1,x_2,x_3)+ \Big(r\sin(\theta)\cos(\phi)\frac{\partial p}{\partial x_1}
             +r\sin(\theta)\sin(\phi)\frac{\partial p}{\partial x_2}
                            +r\cos(\theta)\frac{\partial p}{\partial x_3}  \Big)_{\chi_{[r<1]}}\Bigg]\\
    & ~~~~
                  \frac{m(\theta,\phi)\sin\theta}{r^{1+\beta}} d\phi d\theta dr,
\end{split}
\end{equation*}}
where the probability density function $m(\theta)$ or $m(\theta,\phi)$ specifies the distribution of particles spreading in the radial direction of $\Y$; among them,
$m(\theta)$ is defined on $[0,2\pi]$, satisfying $\int_{0}^{2\pi}m(\theta)d\theta=1$,
while $m(\theta,\phi)$ is defined on a $[0,\pi]\times[0,2\pi]$ rectangular domain, satisfying $\int_0^\pi \int_{0}^{2\pi}m(\theta,\phi)d\phi d\theta=1$.

For the tempered L\'{e}vy flight, we can describe the movement of microscopic particles and derive the macroscopic equations by defining
\begin{equation}
  \nu(d\Y)= \frac{1}{|\Gamma(-\beta)|}\, \frac{m(\Y)}{e^{\lambda|\Y|}|\Y|^{n+\beta}} d\Y;
\end{equation}
and \eqref{aniso1} becomes
\begin{equation}\label{anisoLapTemp}
\begin{split}
  \frac{\partial p(\X,t)}{\partial t} = & \frac{1}{|\Gamma(-\beta)|}\,
  \int_{\mathbb{R}^n\backslash\{0\}} \left[p(\X-\Y)-p(\X)+(\Y\cdot\nabla_\X p(\X))_{\chi_{[|\Y|<1]}} \right]
  \\
  &\cdot \frac{m(\Y)}{e^{\lambda|\Y|}|\Y|^{n+\beta}} d\Y.
 \end{split}
\end{equation}
We write Eqs. \eqref{anisoLap} and \eqref{anisoLapTemp}, respectively, as
%Equations \eqref{anisoLap} and \eqref{anisoLapTemp} will be denoted by
\begin{equation}
  \frac{\partial p(\X,t)}{\partial t} = \Delta_m^{\beta/2} p(\X,t)
\end{equation}
and
\begin{equation}
  \frac{\partial p(\X,t)}{\partial t} = \Delta_m^{\beta/2,\lambda} p(\X,t)
\end{equation}
where the notation $\Delta_m^{\beta/2}$ ($\Delta_m^{\beta/2,\lambda}$) denotes the anisotropic (tempered) fractional Laplacian in $\mathbb{R}^n$; and their definitions are the right hand sides of \eqref{anisoLap} and \eqref{anisoLapTemp}.

\begin{figure}[ht]
\begin{minipage}{0.31\linewidth}
  \centerline{\includegraphics[scale=0.31]{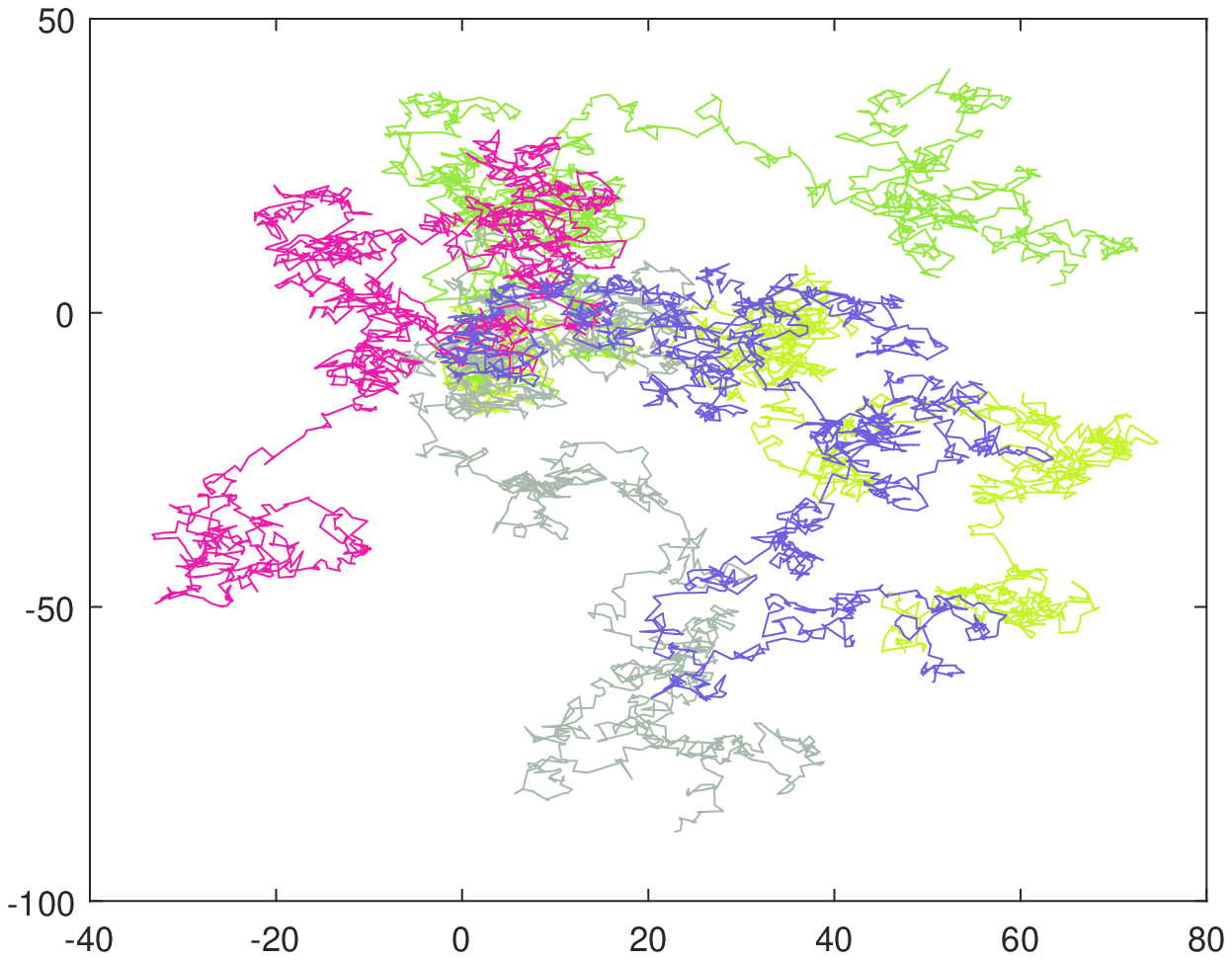}}
  \centerline{(a) $\beta=2$}
\end{minipage}
\hfill
\begin{minipage}{0.31\linewidth}
  \centerline{\includegraphics[scale=0.31]{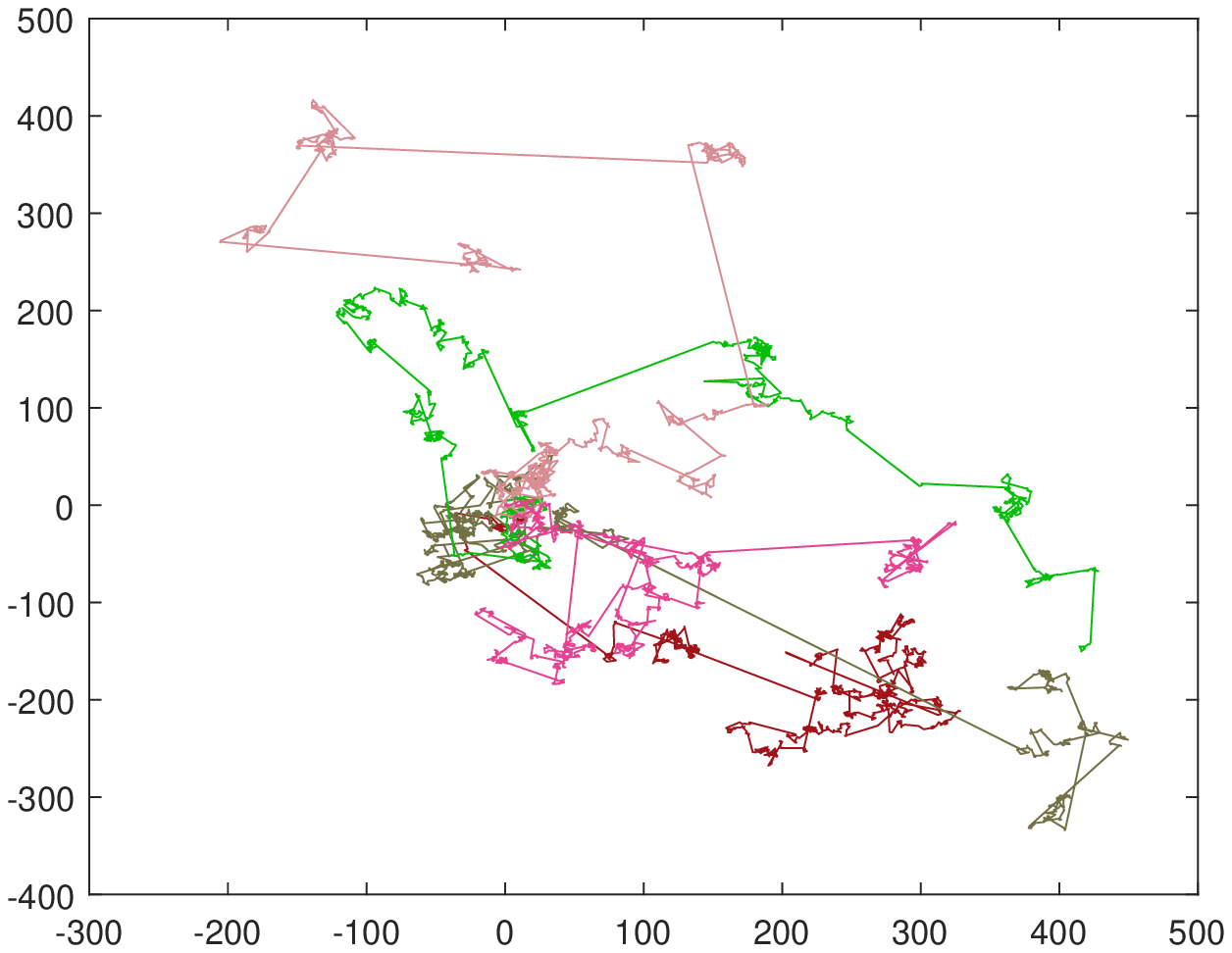}}
  \centerline{(b) $\beta=1.3$}
\end{minipage}
\hfill
\begin{minipage}{0.31\linewidth}
  \centerline{\includegraphics[scale=0.31]{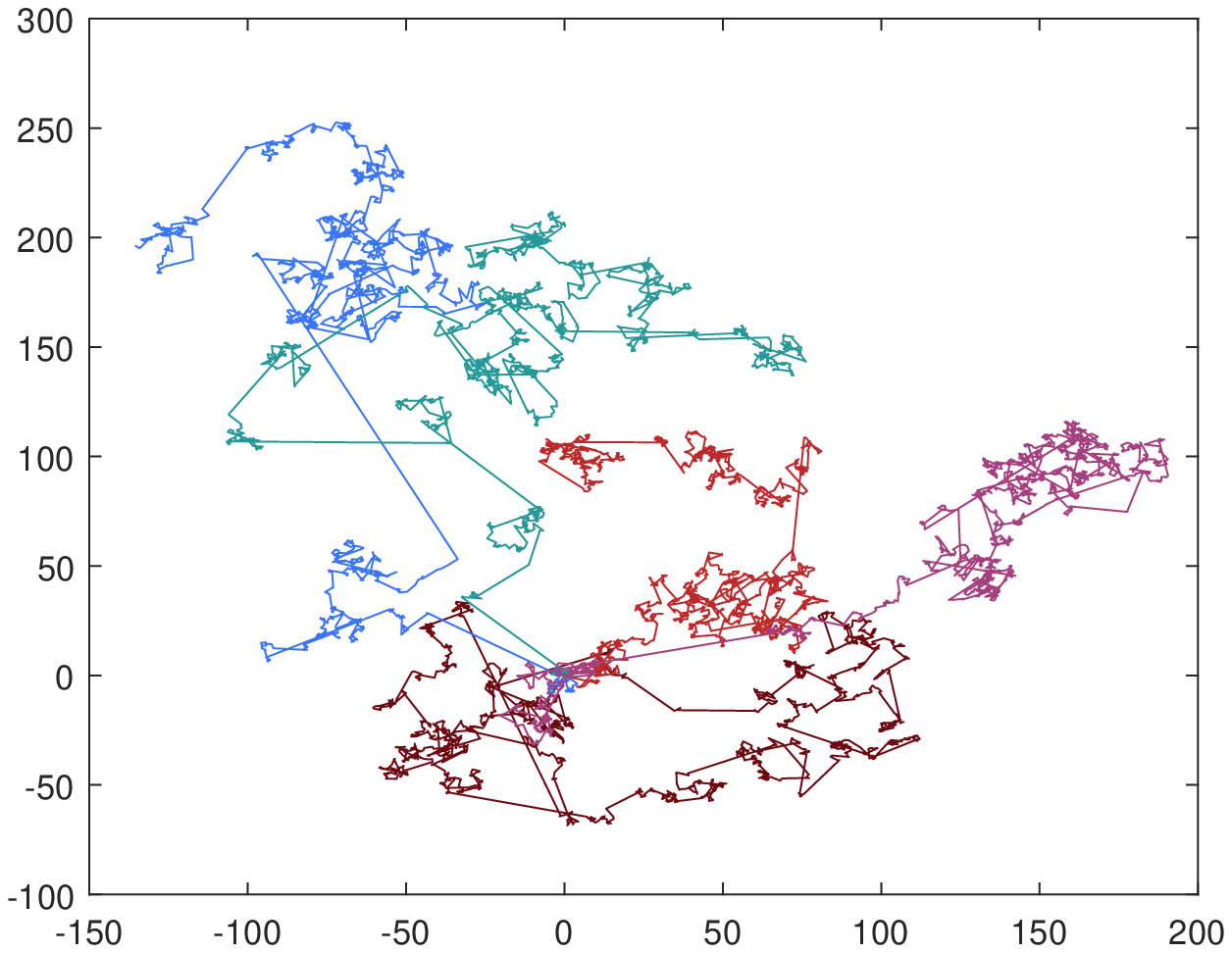}}
  \centerline{(c) $\beta=1.3,\lambda=0.01$}
\end{minipage}
\vfill
\begin{minipage}{0.31\linewidth}
  \centerline{\includegraphics[scale=0.31]{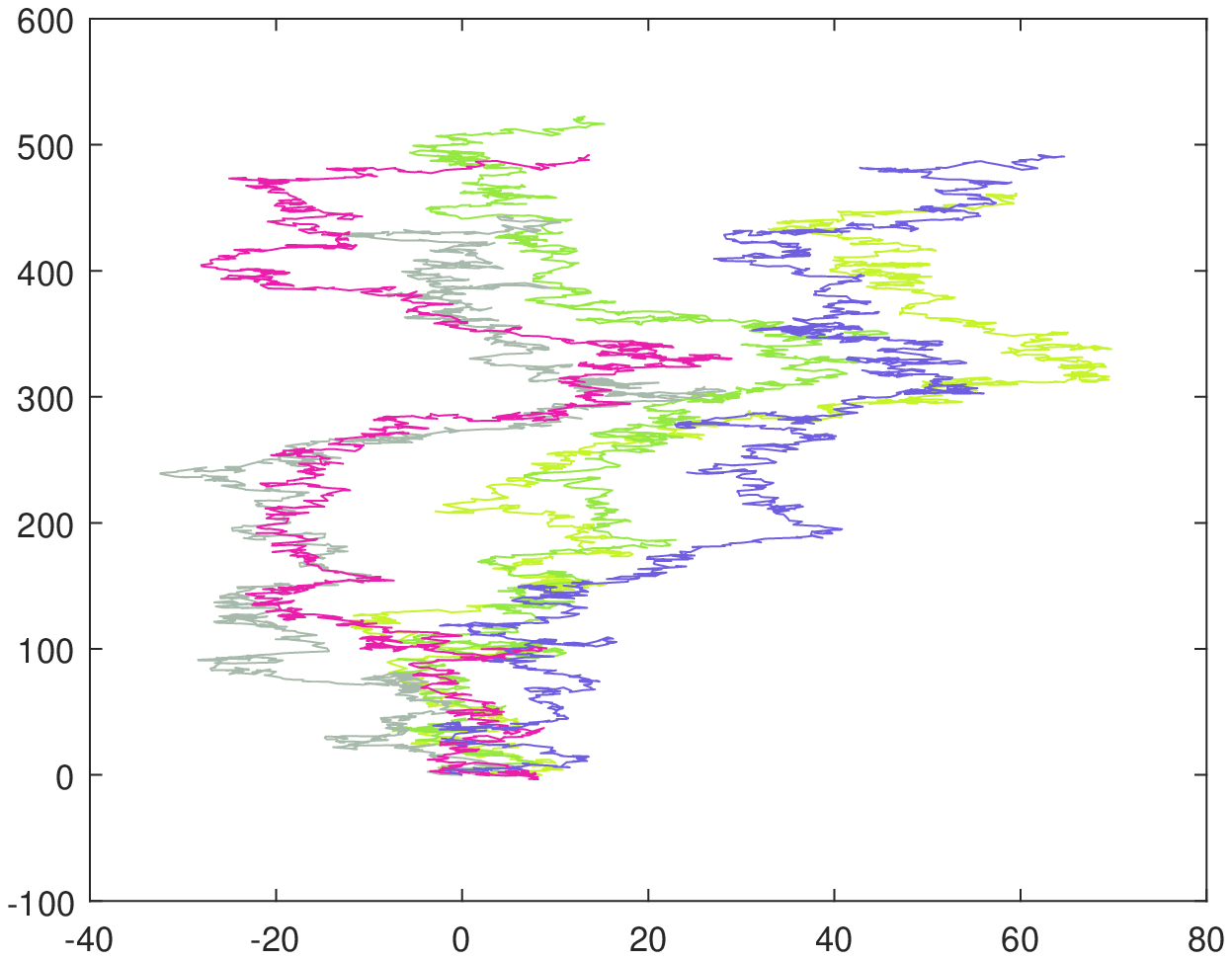}}
  \centerline{(d) $\beta=2$}
\end{minipage}
\hfill
\begin{minipage}{0.31\linewidth}
  \centerline{\includegraphics[scale=0.31]{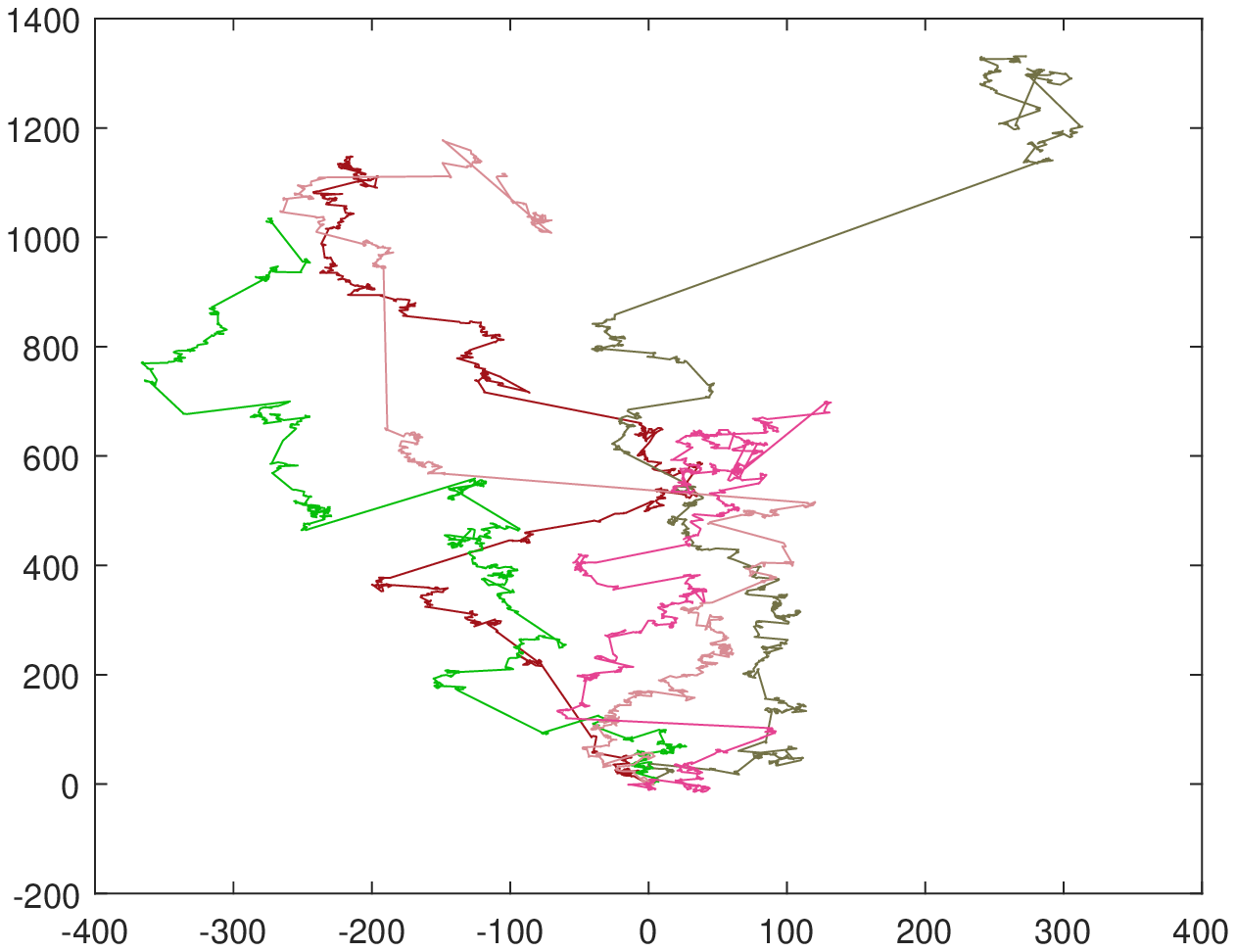}}
  \centerline{(e) $\beta=1.3$}
\end{minipage}
\hfill
\begin{minipage}{0.31\linewidth}
  \centerline{\includegraphics[scale=0.31]{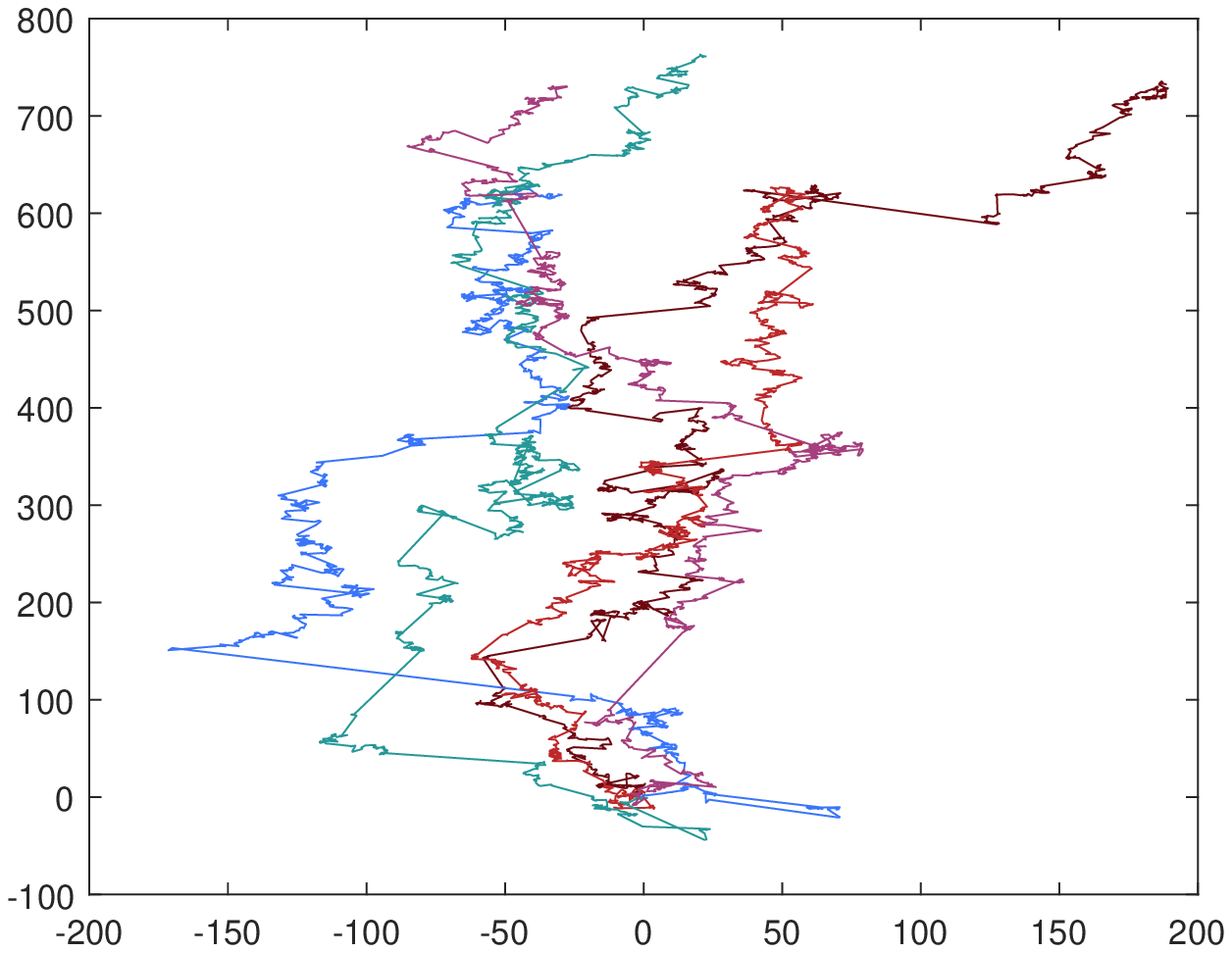}}
  \centerline{(f) $\beta=1.3,\lambda=0.01$}
\end{minipage}
  \caption{Five random trajectories (2000 steps) in each graph with L\'{e}vy laws $\beta=2$ (Gaussian), $\beta=1.3$ (L\'{e}vy flight),  and $\beta=1.3$ and $\lambda=0.01$ (tempered L\'{e}vy flight). For the comparison between the top row and the bottom row, representing isotropic and anisotropic case respectively, the same underlying random number seeds have been used, which is shown by the same color in the same column. }\label{fig1}
\end{figure}

We simulate the trajectories of the particles with the anisotropic movements. Figure \ref{fig1} shows five random trajectories of 2000 steps of L\'{e}vy flight with $\beta=2$ (Gaussian), $\beta=1.3$ and tempered L\'{e}vy flight with $\beta=1.3$ and $\lambda=0.01$ in two dimensions. All trajectories start from the origin $(0,0)$. Three pictures on the top row correspond to the isotropic case, i.e., $m(\Y)=1/(2\pi)$ for $\arg(\Y)\in[0,2\pi)$, while another three on bottom row correspond to the anisotropic case, where we choose $m(\Y)=2/(3\pi)$ for $\arg(\Y)\in(0,\pi)$ and $m(\Y)=1/(3\pi)$ for $\arg(\Y)\in(\pi,2\pi)$.
Note that (a) and (d) depict the isotropic and anisotropic Gaussian jump processes introduced in Section \ref{Sec2}.
By horizontal comparison, the lengths of Gaussian jumps in (a) have almost the same sizes, while L\'{e}vy flight in (b) preforms rare but large jumps. And an exponential truncation in (c) with even little $\lambda=0.01$ excludes large jumps compared with (b). By vertical comparison, in the bottom row, particles are more inclined to move upward and thus finally farther than the isotropic case with the same steps.

Different from \eqref{anisoLap} and \eqref{anisoLapTemp}, an alternative definition of the anisotropic (tempered) fractional Laplacians is given by Fourier transform \cite[Eq. 2]{Meerschaert:99}, with an analogous tempered one presented here:
\begin{equation}\label{anisoLapFT}
  \mathscr{F}[\Delta_m^{\beta/2} p(\X,t)]= (-1)^{\lceil\beta\rceil} \Bigg[ \int_{|\fai|=1}(-i\k\cdot\fai)^\beta m(\fai) d\fai \Bigg]\hat{p}(\k,t)
\end{equation}
and
\begin{equation}\label{anisoLapTempFT}
  \mathscr{F}[\Delta_m^{\beta/2,\lambda} p(\X,t)]= (-1)^{\lceil\beta\rceil} \Bigg[ \int_{|\fai|=1} \left((\lambda-i\k\cdot\fai)^\beta-\lambda^\beta \right) m(\fai) d\fai \Bigg]\hat{p}(\k,t).
\end{equation}
It seems that these definitions are natural for the study of the governing equations,
%These definitions are more suited to the study of the governing equations,
since the symbol $(-i\k\cdot\fai)^\beta$ for $\beta\in(0,1)\cup(1,2)$ denotes $\beta$-order fractional directional derivative. Now we consider the question of when the two ways of defining the operators are equivalent.
To establish the relationship between them, we focus on two cases:
\begin{itemize}
  \item {\bf Case I}: $0<\beta<1$ or $m$ is symmetric. Recall that here the third term in \eqref{anisoLap} and \eqref{anisoLapTemp} can be deleted,
    \begin{equation}\label{anisoLapCase1}
      \Delta_m^{\beta/2} p(\X,t)=
      \frac{1}{|\Gamma(-\beta)|}\int_{\mathbb{R}^n\backslash\{0\}} \left[p(\X-\Y)-p(\X) \right] \frac{m(\Y)}{|\Y|^{n+\beta}} d\Y,
    \end{equation}
    \begin{equation}\label{anisoLapTempCase1}
      \Delta_m^{\beta/2,\lambda} p(\X,t)=
      \frac{1}{|\Gamma(-\beta)|}\int_{\mathbb{R}^n\backslash\{0\}} \left[p(\X-\Y)-p(\X) \right] \frac{m(\Y)}{e^{\lambda|\Y|}|\Y|^{n+\beta}} d\Y.
    \end{equation}
  \item {\bf Case II}: $1<\beta<2$ and $m$ is asymmetric. Recall that here the integrals in \eqref{anisoLap} and \eqref{anisoLapTemp} without the third terms can be understood in the Hadamard sense \cite[(5.65)]{Samko:93}, i.e.,
    \begin{equation}\label{anisoLapCase2}
    \begin{split}
      \Delta_m^{\beta/2} p(\X,t)
      &= \textrm{p.f.}~ \frac{1}{|\Gamma(-\beta)|}\int_{\mathbb{R}^n\backslash\{0\}} \left[p(\X-\Y)-p(\X) \right]
       \frac{m(\Y)}{|\Y|^{n+\beta}} d\Y  \\
      &= \frac{1}{|\Gamma(-\beta)|}\int_{\mathbb{R}^n\backslash\{0\}} \left[p(\X-\Y)-p(\X)+(\Y\cdot\nabla_\X p(\X)) \right]
      \\
      & ~~~~ \cdot \frac{m(\Y)}{|\Y|^{n+\beta}} d\Y ,
    \end{split}
    \end{equation}
    \begin{equation}\label{anisoLapTempCase2}
    \begin{split}
      \Delta_m^{\beta/2,\lambda} p(\X,t)
      &= \textrm{p.f.}~ \frac{1}{|\Gamma(-\beta)|}\int_{\mathbb{R}^n\backslash\{0\}} \left[p(\X-\Y)-p(\X) \right]
       \frac{m(\Y)}{e^{\lambda|\Y|}|\Y|^{n+\beta}} d\Y  \\
      &= \frac{1}{|\Gamma(-\beta)|}\int_{\mathbb{R}^n\backslash\{0\}} \left[p(\X-\Y)-p(\X)+(\Y\cdot\nabla_\X p(\X)) \right]  \\
       & ~~~~ \cdot \frac{m(\Y)}{e^{\lambda|\Y|}|\Y|^{n+\beta}} d\Y-\frac{1}{|\Gamma(-\beta)|}\Gamma(1-\beta)\lambda^{\beta-1}(\mathbf{b}\cdot\nabla_{\X} p(\X)),
    \end{split}
    \end{equation}
    where $\mathbf{b}=\int_{|\fai|=1}\fai\, m(\fai)d\fai$.
\end{itemize}
In Case II, since the high singularity makes the integral divergent, we use the notation $\textrm{p.f.}$ to denote its finite part in the Hadamard sense.

Then we have the following theorem; see Appendix for the proof, which further implies the equality \eqref{anisoLapTempCase2}.
\begin{theorem}\label{theo1}
 Let $m(\Y)$ be any probability density function on the unit sphere and $\lambda\geq0$. The definitions of the anisotropic (tempered) fractional Laplacians $\Delta_m^{\beta/2,\lambda}$ in both Case I and Case II are, respectively, equivalent to $\Delta_m^{\beta/2,\lambda}$ in \eqref{anisoLapFT} and \eqref{anisoLapTempFT} in $\mathbb{R}^n$.
\end{theorem}

%\begin{remark}\label{C_beta}
%  To guarantee the right hand side of \eqref{anisoLapTempFT} recovers the classical case $-|\k|^2 \hat{p}(\k)$ when $\lambda=0,\beta=2$ and $m(\Y)$ is a constant function,
%    the multiplier $\frac{1}{\Gamma(-\beta)}\Gamma(-\beta)$ in Theorem \ref{theo1} should be a bounded constant when $\beta\rightarrow2$.
%\end{remark}

%\begin{remark}\label{anisoRemark1}
%Eased on the definition $\mathscr{F}[g(\X)](\k)=\int_{\mathbb{R}^n}e^{-i\k\cdot\X}g(\X)d\X,$ Eqs. \eqref{anisoLapFT} and \eqref{anisoLapTempFT} can be, respectively, rewritten as
%\begin{equation}\label{anisoLapFT2}
%  \mathscr{F}[\Delta_m^{\beta/2} p(\X,t)]= (-1)^{\lceil\beta\rceil} \Bigg[ \int_{|\fai|=1}(i\k\cdot\fai)^\beta m(\fai) d\fai \Bigg]\hat{p}(\k,t)
%\end{equation}
%and
%\begin{equation}\label{anisoLapTempFT2}
%  \mathscr{F}[\Delta_m^{\beta/2,\lambda} p(\X,t)]= (-1)^{\lceil\beta\rceil} \Bigg[ \int_{|\fai|=1} \left((\lambda+i\k\cdot\fai)^\beta-\lambda^\beta \right) m(\fai) d\fai \Bigg]\hat{p}(\k,t).
%\end{equation}
%\end{remark}

We have just defined the anisotropic (tempered) fractional Laplacian by extending the L\'{e}vy measure $\nu(d\Y)$ with different probability distribution in different directions. More generally, another two variables jump length exponent $\beta$ and truncation exponent $\lambda$ can also be generalized to be anisotropic, i.e., $\beta(\fai)$ and $\lambda(\fai)$, sometimes abused by $\beta(\Y)$ and $\lambda(\Y)$ similar to $m(\Y)$. Let $\beta(\fai)\in(0,1)\cup(1,2)$ and $\lambda(\fai)\geq0$. When  $\lambda(\fai)\equiv0$, it goes back to anisotropic fractional Laplacian. Following \eqref{anisoLapFT}, \eqref{anisoLapTempFT}, \eqref{anisoLapTempCase1} and \eqref{anisoLapTempCase2}, the definitions of new anisotropic (tempered) fractional Laplacian are, respectively,
\begin{itemize}
  \item Case I: $0<\beta<1$ or $m$ is symmetric,
    \begin{equation}\label{anisoLapTemp2Case1}
    \begin{split}
      \tilde\Delta_m^{\beta/2,\lambda} p(\X,t)&=
      \int_{\mathbb{R}^n\backslash\{0\}} \left[p(\X-\Y)-p(\X) \right]
      \\
      &~~~~\cdot \frac{m(\Y)}{|\Gamma(-\beta(\Y))| e^{\lambda(\Y)|\Y|}|\Y|^{n+\beta(\Y)}} d\Y.
      \end{split}
    \end{equation}
  \item Case II: $1<\beta<2$ and $m$ is asymmetric,
    \begin{equation}\label{anisoLapTemp2Case2}
    \begin{split}
      \tilde\Delta_m^{\beta/2,\lambda} p(\X,t)
      &= \textrm{p.f.}~ \int_{\mathbb{R}^n\backslash\{0\}} \left[p(\X-\Y)-p(\X) \right]
      \\
      &~~~~\cdot
       \frac{m(\Y)}{|\Gamma(-\beta(\Y))|e^{\lambda(\Y)|\Y|}|\Y|^{n+\beta(\Y)}} d\Y  \\
      &= \int_{\mathbb{R}^n\backslash\{0\}} \left[p(\X-\Y)-p(\X)+(\Y\cdot\nabla_\X p(\X)) \right] \\
       & ~~~~  \cdot\frac{m(\Y)}{|\Gamma(-\beta(\Y))|e^{\lambda(\Y)|\Y|}|\Y|^{n+\beta(\Y)}} d\Y-(\mathbf{b}\cdot\nabla_{\X} p(\X)),
    \end{split}
    \end{equation}
    where $\mathbf{b}=\int_{|\fai|=1}\Gamma(1-\beta(\fai)) \,\lambda(\fai)^{\beta(\fai)-1} \,\fai\, m(\fai)/|\Gamma(-\beta(\fai))|d\fai$.
\end{itemize}
In Fourier space, the new operator has the form
\begin{equation}
  \mathscr{F}[\tilde\Delta_m^{\beta/2,\lambda}p(\X,t)] =(-1)^{\lceil\beta\rceil} \int_{|\fai|=1}\Big( (\lambda(\fai)-i\k\cdot\fai)^{\beta(\fai)}  - \lambda(\fai)^{\beta(\fai)} \Big)~m(\fai)d\fai \hat{p}(\k,t).
\end{equation}

\begin{figure}[ht]
\begin{minipage}{0.31\linewidth}
  \centerline{\includegraphics[scale=0.31]{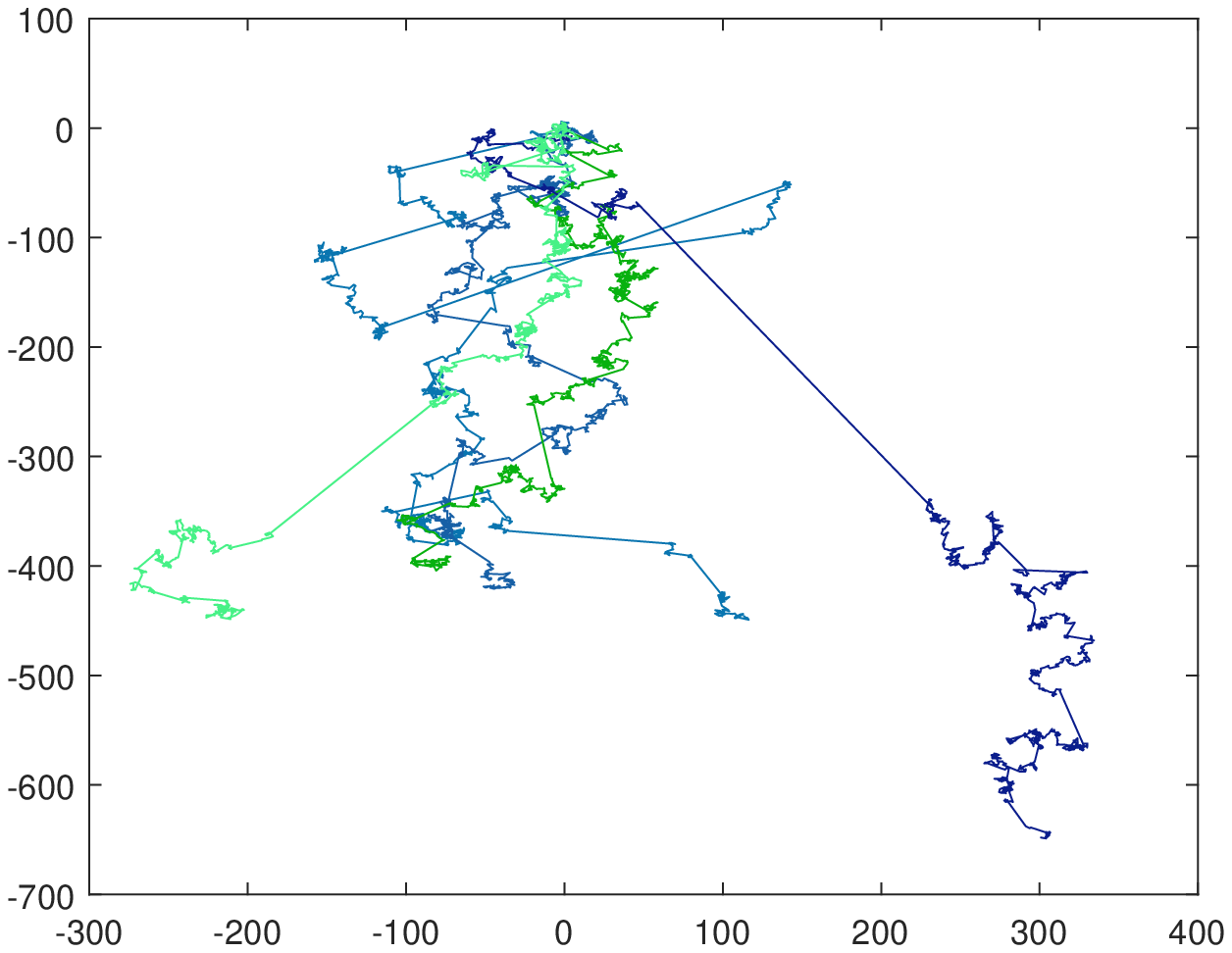}}
  \centerline{(a)}
\end{minipage}
\hfill
\begin{minipage}{0.31\linewidth}
  \centerline{\includegraphics[scale=0.31]{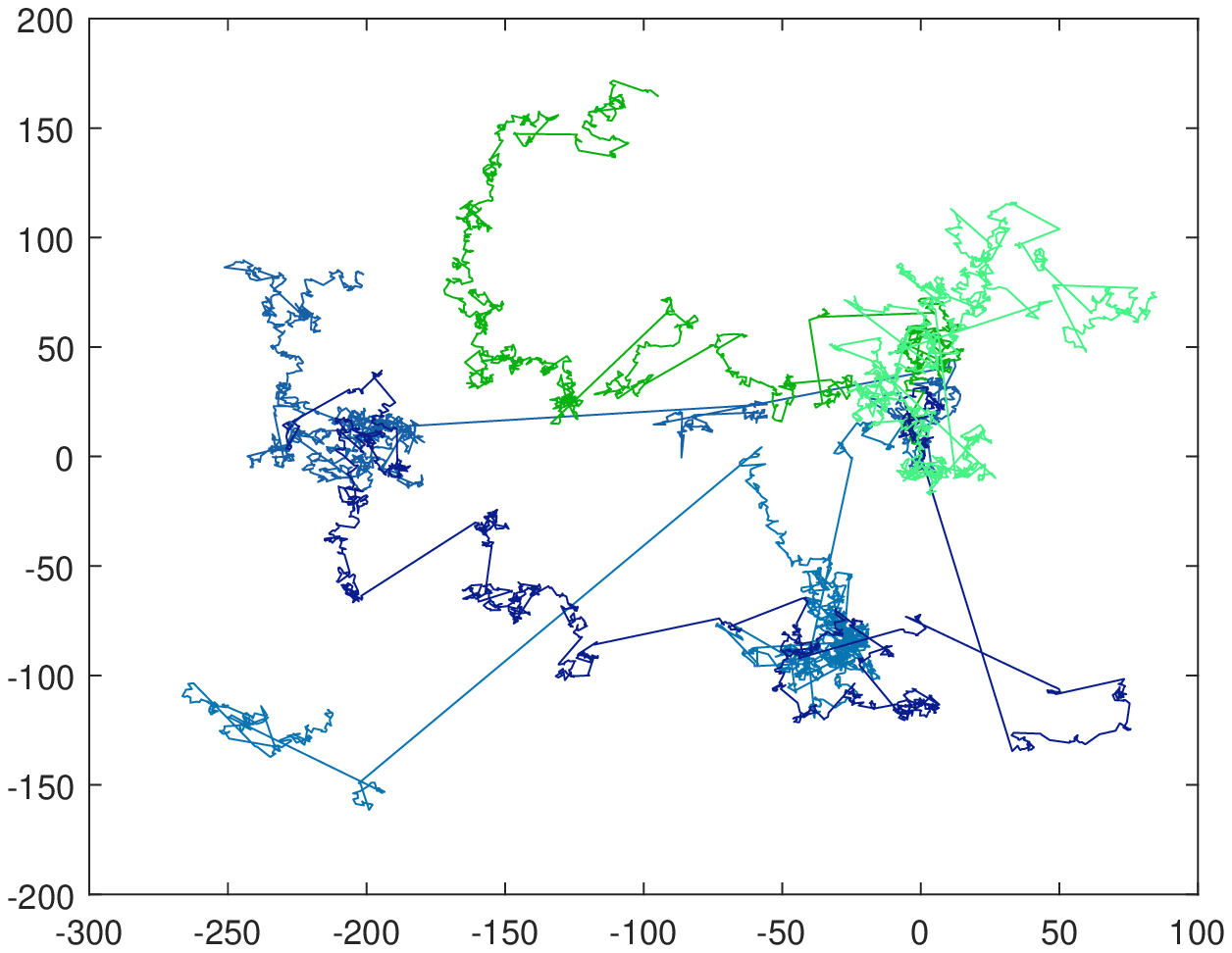}}
  \centerline{(b)}
\end{minipage}
\hfill
\begin{minipage}{0.31\linewidth}
  \centerline{\includegraphics[scale=0.31]{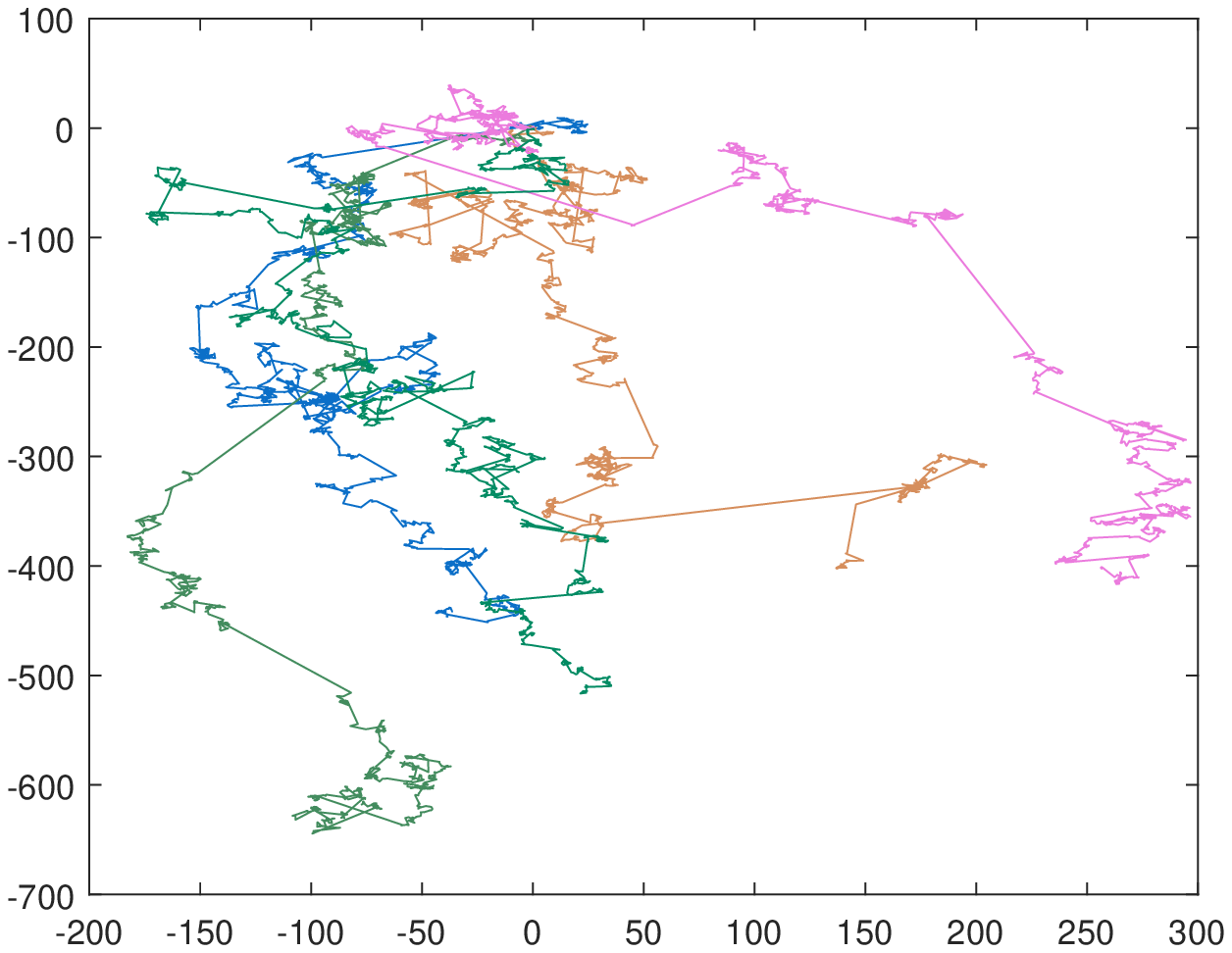}}
  \centerline{(c)}
\end{minipage}
\caption{Five random trajectories (2000 steps) in every graph with more general anisotropic L\'{e}vy measure. $(a)$: $\beta=1.8$ for $\textrm{arg}(\fai)\in(0,\pi)$, and $\beta=1.4$ for $\textrm{arg}(\fai)\in(\pi,2\pi)$; $(b)$: $\beta=1.8$ and  $m(\fai)=0.6/\pi$ for $\textrm{arg}(\fai)\in(0,\pi)$, and $\beta=1.4$ and $m(\fai)=0.4/\pi$ for $\textrm{arg}(\fai)\in(\pi,2\pi)$; $(c)$: $\beta=1.3$ and  $\lambda=0.01$ for $\textrm{arg}(\fai)\in(0,\pi)$, and $\beta=1.3$ and $\lambda=0$ for $\textrm{arg}(\fai)\in(\pi,2\pi)$.
To compare $(a)$ and $(b)$, the same underlying random number seeds (the trajectories of the same color) are used here.}\label{fig1.2}
\end{figure}

We also simulate the trajectories of particles with the new anisotropic L\'{e}vy measure $\nu(d\Y)$ defined in \eqref{anisoLapTemp2Case1}. As
Figure \ref{fig1.2} shows, we take the isotropic $m(\fai), \lambda(\fai)=0$ and the anisotropic $\beta(\fai)$ to be $1.8$ for $\textrm{arg}(\fai)\in(0,\pi)$ and $1.4$ for $\textrm{arg}(\fai)\in(\pi,2\pi)$ in $(a)$; the particles move farther downward than upward. In $(b)$, only difference with the parameter in $(a)$ is the anisotropic $m(\fai)$ being $0.6/\pi$ for $\textrm{arg}(\fai)\in(0,\pi)$ and $0.4/\pi$ for $\textrm{arg}(\fai)\in(\pi,2\pi)$. This choice of $m(\fai)$ aims to balance the anisotropic $\beta(\fai)$; as the second graph shows, the movements of particles become almost isotropic. In $(c)$, we take the isotropic $m(\fai)$ and $\beta(\fai)$, but the anisotropic $\lambda(\fai)$ to be $0.01$ for $\textrm{arg}(\fai)\in(0,\pi)$ and $0$ for $\textrm{arg}(\fai)\in(\pi,2\pi)$; the particles move farther downward than upward.

\begin{remark}
%  People always use Fick's law to describe diffusion phenomenon.  It postulates that the diffusion flux goes from regions of high concentration to regions of low concentration, with a magnitude that is proportional to the concentration gradient, i.e., $J=-D\frac{d\phi}{dx}$.
In the practical problem, the directional measure may depend on
concentration gradient.
%the diffusion measure
%
%  Inspired by this,
To emphasize the effects caused by the directional gradient, the definition of the anisotropic (tempered) fractional Laplacian in \eqref{anisoLapTemp2Case1} can be extended to
  \begin{equation}
  \begin{split}
   \tilde\Delta_m^{\beta/2,\lambda}p(\X,t)=& (-1)^{\lceil\beta\rceil}\int_{\mathbb{R}^n\backslash\{0\}} \left[p(\X-\Y)-p(\X) \right]
  \\
  &
    \cdot\frac{m\left(\Y,\frac{\partial p(\Y)}{\partial\Y}\right)}{|\Gamma(-\beta(\Y))|e^{\lambda(\Y)|\Y|}|\Y|^{n+\beta(\Y)}}  d\Y,
   \end{split}
  \end{equation}
  where $m$ should be an increasing function of directional gradient $\frac{\partial p(\Y)}{\partial\Y}$.
\end{remark}

As a complement to the definition of the anisotropic (tempered) fractional Laplacian \eqref{anisoLapFT} and \eqref{anisoLapTempFT}, we also present the definition of the operator in the case that $\beta=1$, i.e., let $\nu(d\Y)= \frac{m(\Y)}{|\Y|^{n+1}} d\Y$, which still is a nonlocal operator. For the sake of simplicity, we assume that $m(\Y)$ is symmetric, then the term $(\Y\cdot\nabla_\X p(\X))_{\chi_{[|\Y|<1]}}$ in \eqref{aniso1} can be omitted. For the one dimensional asymmetric operators with $\beta=1$, see \cite{Kelly:17} for the details.

%See Zolotarev fractional derivative for $\beta=1$ in \cite{Kelly:17}, where the asymmetric operators in one dimension are discussed in detail.
\begin{proposition}
  Let $\beta=1$ and $\lambda>0$. If the probability density function $m(\Y)$ is symmetric, then the Fourier symbols of the anisotropic fractional Laplacian and the corresponding tempered one, respectively, are
  \begin{equation}\label{anisoLapEqual1}
    \mathscr{F}[\Delta_m^{1/2}p(\X,t)]=   \,\frac{\pi}{2} \int_{|\fai|=1} |(\k\cdot\fai)| ~m(\fai)d\fai \cdot \hat{p}(\k,t)
  \end{equation}
  and
  \begin{equation}\label{anisoLapTempEqual1}
  \begin{split}
    \mathscr{F}[\Delta_m^{1/2,\lambda}p(\X,t)]= & \int_{|\fai|=1} \Big[(\k\cdot\fai)\arctan\left(\frac{\k\cdot\fai}{\lambda}\right) -\frac{\lambda}{2}\ln(\lambda^2+(\k\cdot\fai)^2) \\
    &
    +\lambda\ln\lambda\Big]m(\fai)d\fai \cdot \hat{p}(\k,t).
    \end{split}
  \end{equation}
\end{proposition}
\begin{proof}
  We firstly prove the tempered case. Taking the Fourier transform of the right hand side of \eqref{anisoLapTemp}, we have
  \begin{equation*}
  \begin{split}
     \mathscr{F}\Big[\Delta_m^{1/2,\lambda} p(\X,t)\Big](\k)
     &= \int_{\mathbb{R}^n}\frac{e^{i\k\cdot\Y}-1}{e^{\lambda|\Y|}|\Y|^{n+1}}m(\Y)d\Y \cdot\hat{p}(\k,t)  \\
     &= \Bigg[\int_{\mathbb{R}^n}\frac{\cos(\k\cdot\Y)-1}{e^{\lambda|\Y|}|\Y|^{n+1}}m(\Y)d\Y \Bigg]  \cdot\hat{p}(\k,t),
  \end{split}
  \end{equation*}
  where the term $i\sin(\k\cdot\Y) $ vanishes due to the symmetry of $m(\Y)$.
  By polar coordinate transformation and integration by parts, we have
  \begin{equation*}
    \begin{split}
      \int_{\mathbb{R}^n}\frac{1-\cos(\k\cdot\Y)}{e^{\lambda|\Y|}|\Y|^{n+1}}m(\Y)d\Y
      =&\int_0^\infty\int_{|\fai|=1} r^{-2}e^{-\lambda r}(1-\cos(r\k\cdot\fai))m(\fai)d\fai dr \\
      =&-\lambda^2\int_0^\infty \ln(r)e^{-\lambda r} \int_{|\fai|=1} (1-\cos(r\k\cdot\fai))m(\fai)d\fai dr \\
      &+2\lambda  \int_0^\infty \ln(r)e^{-\lambda r} \int_{|\fai|=1}(\k\cdot\fai) \sin(r\k\cdot\fai) m(\fai)d\fai dr \\
       &         -\int_0^\infty \ln(r)e^{-\lambda r} \int_{|\fai|=1} (\k\cdot\fai)^2 \cos(r\k\cdot\fai) m(\fai)d\fai dr,
    \end{split}
  \end{equation*}
  from which Eq. \eqref{anisoLapTempEqual1} can be directly obtained by using \cite[Eq. 4.441(1-2)]{Gradshteyn:80}.

  For the proof of \eqref{anisoLapEqual1}, taking $\lambda=0$ in \eqref{anisoLapTempEqual1} leads to
  \begin{equation*}
    \begin{split}
      \mathscr{F}[\Delta_m^{1/2}p(\X,t)]
      &= \int_{|\fai|=1} (\k\cdot\fai) \frac{\pi}{2} \textrm{sgn}(\k\cdot\fai) ~m(\fai)d\fai \cdot \hat{p}(\k,t)  \\
      &= \frac{\pi}{2} \int_{|\fai|=1} |(\k\cdot\fai)| ~m(\fai)d\fai\cdot\hat{p}(\k,t).
    \end{split}
  \end{equation*}
\end{proof}
  Furthermore, if $m(\fai)$ is isotropic, then
  \begin{equation*}
    \begin{split}
      \mathscr{F}[\Delta_m^{1/2}p(\X,t)]
      &= \frac{\pi}{2\omega_n}\int_{|\fai|=1} |(\k\cdot\fai)| ~d\fai\cdot\hat{p}(\k,t)
       = \frac{\pi}{2\omega_n} |\k| \int_{|\fai|=1} |\cos(\theta_1)| ~d\fai \cdot\hat{p}(\k,t)  \\
      &= \frac{\pi}{2\omega_n} C_n |\k|  \int_0^\pi \sin^{n-2}(\theta_1) |\cos(\theta_1)| d\theta_1 \cdot\hat{p}(\k,t)\\
      &= \frac{1}{\omega_n} \frac{\pi}{n-1}C_n |\k|\cdot\hat{p}(\k,t)
       = \frac{1}{\omega_n}\frac{\pi^{\frac{n+1}{2}}}{\Gamma(\frac{n+1}{2})} |\k|\cdot\hat{p}(\k,t),
    \end{split}
  \end{equation*}
  where $\omega_n$ is the measure of the $n$ dimensional unit sphere, $\omega_n=2\pi^{n/2}/\Gamma(n/2)$ if $n\geq2$ and $\omega_n=2$ when $n=1$; the rotation invariance \cite[Proposition 3.3]{Nezza:12} of the integrand is used in the second equality, and $\cos(\theta_1)$ denotes one of the components of vector $\fai$,
  \begin{equation*}
    C_n=\Big(\int_0^{\pi}\sin^{n-3}(\theta_2)d\theta_2\Big)\cdots\Big(\int_0^{\pi}\sin(\theta_{n-2})d\theta_{n-2}\Big)\Big(\int_0^{2\pi}d\theta_{n-1}\Big)
     =\frac{2\pi^{\frac{n-1}{2}}}{\Gamma(\frac{n-1}{2})}.
  \end{equation*}
Following \eqref{anisoLapTempEqual1}, the Fourier symbol of the new anisotropic tempered fractional Laplacian when $\beta=1$ is
{\small
\begin{equation*}
    \mathscr{F}[\tilde\Delta_m^{1/2,\lambda}]= \int_{|\fai|=1} \Big[(\k\cdot\fai)\arctan\left(\frac{\k\cdot\fai}{\lambda(\fai)}\right) -\frac{\lambda(\fai)}{2}\ln(\lambda(\fai)^2+(\k\cdot\fai)^2)+\lambda(\fai)\ln(\lambda(\fai)) \Big] ~m(\fai)d\fai.
\end{equation*}}

All the discussions above are based on compound Poisson processes with different probability distribution of jump length for (tempered) L\'{e}vy flights, which render the deterministic governing equations with classical first derivative temporally. If instead, the fractional Poisson processes are taken as the renewal process, in which the time interval between each pair of events follows the power law distribution. Then the deterministic governing equations with Caputo fractional derivative temporally can be derived. More precisely, let $S(t)$ be a nondecreasing subordinator \cite{Chen:05} with Laplace exponent $s^\alpha$, $\alpha\in(0,1)$. Then consider a new process $\Z(t)=\X(E(t))$, where $\X(t)$ is the L\'{e}vy process discussed in \eqref{Levyf_kspace} with Fourier symbol $\Phi_0(\k)-1$ and the inverse subordinator $E(t)=\inf\{\tau>0:S(\tau)>t\}$. Then similarly to \cite[Eq. (16)-(17)]{Deng:17}, we have
\begin{equation*}
  p_z(\Z,t)=\int_0^\infty p_x(\Z,\tau)p_e(\tau,t)d\tau,
\end{equation*}
where $p_e(\tau,t)$ denotes the PDF of $E(t)$. Performing the Fourier-Laplace transform leads to
\begin{equation*}
  \tilde{\hat{p}}_z(\k,s)= \frac{s^{\alpha-1}}{s^\alpha+1-\Phi_0(\k)},
\end{equation*}
where the notation $\tilde{\cdot}$ denotes the Laplace transform from $t$ to $s$.
Arranging the terms and performing the inverse Laplace transform, one obtains
\begin{equation}\label{Timefrac}
  {}_0^C D_t^\alpha \hat{p}_z(\k,t)=(\Phi_0(\k)-1)\hat{p}_z(\k,t),
\end{equation}
the only difference of which with \eqref{Levyf_kspace} is the temporal derivative. Then, as the way of treating \eqref{Levyf_kspace}, taking the inverse Fourier transform results in the corresponding deterministic equations, the specific expressions of which depend on the different $\nu(d\Y)$.

\section{Multiple internal states with anisotropic diffusion}\label{Sec4}

Now, we derive the fractional Fokker-Planck and Feymann-Kac equations with multiple internal states, being both temporal and spatial, with the spatial operators being the anisotropic (tempered) fractional Laplacian $\Delta_m^{\beta/2,\lambda}$ presented in the above section.
We first try to make it clear what multiple internal states mean. By CTRW models, the motion of particles is characterised by two random variables, i.e., waiting time $\xi$ and jump length $\eta$. Assume the process only has three different possibilities of distributions of $\xi$ and/or $\eta$ at each step. We call it three internal states $S1$, $S2$ and $S3$, as in Figure \ref{fig2}. The information contained in each internal state $Si\, (i=1,2,3)$ is the distributions of $\xi$ and $\eta$ at current step. More general models may contain more information and more internal states. In one step, each possibility of the three will yield the next step still with three different possibilities. So step after step, a Markov chain is formed. As long as the initial distribution $|\textrm{init}\rangle$ and transition matrix $M$ are given, the distribution of internal states of $n$-th step can be easily obtained, denoted by $(M^T)^{n-1}|\textrm{init}\rangle$. Here, the element $m_{ij}$ of the matrix $M$ denotes the transition probability from state $i$ to state $j$, and the notations bras $\langle\cdot|$ and kets $|\cdot\rangle$ denote the row and column vectors, respectively.

\begin{figure}[ht]
  \centering
  \includegraphics[scale=0.64]{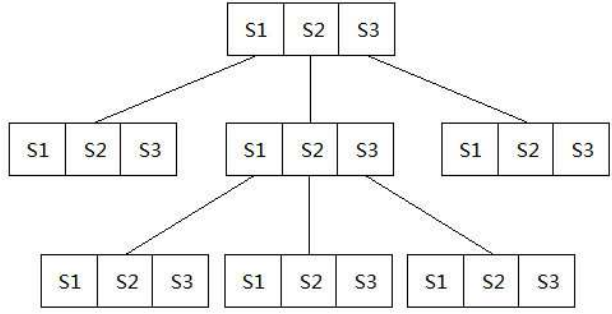}\\
  \caption{Three internal states in each step. Each internal state of $S1$, $S2$ and $S3$ contains different distributions of waiting time $\xi$ and/or jump length $\eta$.}
  \label{fig2}
\end{figure}

The number of the internal states are taken as $N$ for fractional Fokker-Planck and Feymann-Kac equations,
%Now we derive the fractional Fokker-Planck and Feymann-Kac equations with $N$ internal states.
the derivation processes of which are similar to the ones given in \cite{Xu:17}. Here we only provide the derivation of Fokker-Planck equation.
We denote the column vector by capital letter and its components by lowercase letters, e.g., $|G(\X,t)\rangle$ with its components $g^i(\X,t), i=1,2\cdots,N$ being the PDF of finding the particle, at time $t$, position $\X$ in $n$ dimensional space and internal state $i$. Then define the waiting time distribution matrix $\Phi(t)=\textrm{diag}(\phi^1(t),\phi^2(t),\cdots,\phi^N(t))$ and the jump length one $\Lambda(\X)=\textrm{diag}(\lambda^1(\X),\lambda^2(\X),\cdots,\lambda^N(\X))$, where $\phi^i(t)$ and $\lambda^i(\X)$ are, respectively, the PDFs of waiting time and jump length at the $i$-th internal state.

Let $|Q_n(\X,t)\rangle$ be composed by $q_n^i(\X,t), i=1,2\cdots,N$, representing the PDF of the particle that just arrives at position $\X$, time $t$, and $i$-th internal state, after $n$ steps. Thus the matrix of survival probability is
\begin{equation*}
\begin{split}
    W(t)&=\textrm{diag}\left(w^1(t),\cdots,w^N(t)\right)  \\
    &=\textrm{diag}\left(\int_t^\infty \phi^1(\tau)d\tau,\cdots,\int_t^\infty \phi^N(\tau)d\tau\right)=I-\int_0^t \Phi(\tau)d\tau,
\end{split}
\end{equation*}
where $I$ denotes the identity matrix. This indicates that the Laplace transform of $W(t)$ is
\begin{equation*}
  \tilde{W}(s)=\frac{I- \tilde{\Phi}(s)}{s}.
\end{equation*}
For $G$ and $Q$, there exists
\begin{equation}\label{GandQ}
  |G(\X,t)\rangle=\int_0^t W(\tau)\sum_{n=0}^\infty |Q_n(\X,t)\rangle d\tau.
\end{equation}
On the other hand, for each component $q_n^i$ of $Q_n$, we have
\begin{equation*}
  q_n^i(\X,t)=\sum_{j=1}^N \int_0^t\int_{\mathbb{R}^n}m_{ji} \Lambda(\X-\Y)\Phi(t-\tau) q_{n-1}^j(\Y,\tau)d\Y d\tau.
\end{equation*}
Thus $Q$ satisfies
\begin{equation}\label{QandQ}
  |Q_n(\X,t)\rangle = \int_0^t\int_{\mathbb{R}^n} M^T \Lambda(\X-\Y)\Phi(t-\tau) |Q_{n-1}(\X,t)\rangle d\Y d\tau.
\end{equation}
Taking Fourier-Laplace transform to \eqref{GandQ} and \eqref{QandQ} leads to
\begin{equation}\label{Montroll}
  |\tilde{\hat{G}}(\k,s)\rangle= \frac{I-\tilde{\Phi}(s)}{s} [I-M^T \hat{\Lambda}(\k) \tilde{\Phi}(s)]^{-1} |\textrm{init}\rangle.
\end{equation}

The Fokker-Planck equation can be obtained by applying inverse Fourier-Laplace transform to \eqref{Montroll}. Here, we take the waiting time distributions as asymptotic power laws, i.e., in Laplace space $\tilde{\Phi}(s)\sim I-\textrm{diag}(s^{\alpha_1},\cdots,s^{\alpha_N}), 0<\alpha_1,\cdots,\alpha_N<1$. As for jump lengths, they obey the L\'{e}vy distributions, i.e., in Fourier space, each component of $\hat{\Lambda}(\k)$ is the form of \eqref{anisoLapTempFT} with particular $\beta_i$ and $\lambda_i$. Then, the Fokker-Planck equation with $N$ internal states is
\begin{equation}
\begin{split}
    M^T \frac{\partial}{\partial t} |G(\X,t)\rangle =&
  ~ (M^T-I) \textrm{diag}(D_t^{1-\alpha_1},\cdots,D_t^{1-\alpha_N}) |G(\X,t)\rangle \\
    &+ M^T \textrm{diag}(D_t^{1-\alpha_1}\Delta_m^{\beta_1/2,\lambda_1},\cdots,D_t^{1-\alpha_N}\Delta_m^{\beta_N/2,\lambda_N})  |G(\X,t)\rangle,
\end{split}
\end{equation}
where $D_t^{1-\alpha_i}$ is the Riemann-Liouville derivative defined as \cite{Podlubny:99}
\begin{equation}\label{RLderi}
  D_t^{1-\alpha_i} g^i(\X,t)= \frac{1}{\Gamma(\alpha_i)}\frac{\partial}{\partial t}\int_0^t\frac{g^i(\X,\tau)}{(t-\tau)^{1-\alpha_i}} d\tau,
\end{equation}
and $\Delta_m^{\beta_i/2,\lambda_i}$ denotes the anisotropic (tempered) fractional Laplacian with its Fourier transform $\hat{\lambda}^i(\k)$.

For Feymann-Kac equations, we define the functional $A=\int_0^t U(\X(\tau))d\tau$, where $U$ is a prespecified function. Denote $G(\X,A,t)$ to be the PDF of the functional $A$ and position $\X$ and $\bar{G}(\X,\rho,t)$ be the Fourier transform from $A$ to $\rho$. Then we directly have the Feymann-Kac equation of the forward version
{\small
\begin{equation}
  \begin{split}
    M^T &\frac{\partial}{\partial t} |\bar{G}(\X,\rho,t)\rangle =
  ~ (M^T-I) \textrm{diag}(\mathcal{D}_t^{1-\alpha_1},\cdots,\mathcal{D}_t^{1-\alpha_N}) |\bar{G}(\X,\rho,t)\rangle \\
    &+ M^T  \textrm{diag}(\Delta_m^{\beta_1/2,\lambda_1}\mathcal{D}_t^{1-\alpha_1},\cdots,\Delta_m^{\beta_N/2,\lambda_N}\mathcal{D}_t^{1-\alpha_N})  |\bar{G}(\X,\rho,t)\rangle
    +i\rho U(\X) M^T |\bar{G}(\X,\rho,t)\rangle,
\end{split}
\end{equation}}
where
\begin{equation}
  \mathcal{D}_t^{1-\alpha_i}\bar{g}^i(\X,\rho,t)= \frac{1}{\Gamma(\alpha_i)}\left(\frac{\partial}{\partial t}-i\rho U(\X)\right)
                                            \int_0^t \frac{e^{i(t-\tau)\rho U(\X)}}{(t-\tau)^{1-\alpha_i}} \bar{g}^i(\X,\rho,\tau) d\tau;
\end{equation}
and the backward version is
{\small
\begin{equation}
  \begin{split}
    M^T &\frac{\partial}{\partial t} |\bar{G}_{\X_0}(\rho,t)\rangle =
  ~ (M^T-I) \textrm{diag}(\mathcal{D}_t^{1-\alpha_1},\cdots,\mathcal{D}_t^{1-\alpha_N}) |\bar{G}_{\X_0}(\rho,t)\rangle \\
    &+ M^T  \textrm{diag}(\mathcal{D}_t^{1-\alpha_1}\Delta_{m,\X_0}^{\beta_1/2,\lambda_1},\cdots,\mathcal{D}_t^{1-\alpha_N}\Delta_{m,\X_0}^{\beta_N/2,\lambda_N})  |\bar{G}_{\X_0}(\rho,t)\rangle
    +i\rho U(\X_0) M^T |\bar{G}_{\X_0}(\rho,t)\rangle.
\end{split}
\end{equation}}

\section{Generalized boundary conditions}\label{Sec5}

In this section, we mainly consider the initial and boundary value problems with the anisotropic tempered fractional Laplacian. The case for the anisotropic fractional Laplacian can be obtained by taking $\lambda=0$. Following the ideas of \cite{Dybiec:06,Deng:17}, the local boundary $\partial\Omega$ itself can not be hit by the majority of discontinuous sample trajectories; based on this physical implication, these problems should be specified the generalized Dirichlet and Neumann type boundary conditions.
For the sake of simplicity, we only discuss the anisotropic tempered fractional Laplacian $\Delta_m^{\beta/2,\lambda}p(\X,t)$ defined in \eqref{anisoLapTempCase1}, i.e., $\lambda$ and $\beta$ are constant,
\begin{equation}\label{anisoLapTempFinal}
 \Delta_m^{\beta/2,\lambda}p(\X,t)= \frac{1}{|\Gamma(-\beta)|}\,
  \int_{\mathbb{R}^n\backslash\{0\}} \left[p(\X-\Y)-p(\X) \right] \frac{m(\Y)}{e^{\lambda|\Y|}|\Y|^{n+\beta}} d\Y.
\end{equation}
Consider the time dependent Dirichlet problem:
\begin{equation}\label{IBFracLap1}
\left\{
\begin{aligned}
& \frac{\partial p(\X,t)}{\partial t}-\Delta_m^{\beta/2,\lambda}p(\X,t)=f(\X,t)  &&\textrm{in}~~ \Omega,  \\
& p(\X,t)=g(\X,t)   &&\textrm{in}~~\mathbb{R}^n \backslash \Omega, \\[3pt]
& p(\X,0)=p_0(\X)   &&\textrm{in}~~\Omega;
\end{aligned} \right.
\end{equation}
and the Neumann problem:
\begin{equation}\label{IBFracLap2}
\left\{
\begin{aligned}
& \frac{\partial p(\X,t)}{\partial t}-\Delta_m^{\beta/2,\lambda}p(\X,t)=f(\X,t)  &&\textrm{in}~~ \Omega,  \\
& \Delta_m^{\beta/2,\lambda}p(\X,t)=g(\X,t)   &&\textrm{in}~~\mathbb{R}^n \backslash \Omega, \\[5pt]
& p(\X,0)=p_0(\X)   &&\textrm{in}~~\Omega.
\end{aligned} \right.
\end{equation}

\begin{remark}
  If we consider the model with Caputo fractional derivative in time, like \eqref{Timefrac}, its Dirichlet problem can be similarly formulated as above while its Neumann problem should be
\begin{equation}\label{IBFracLap3}
\left\{
\begin{aligned}
& {}_0^C D_t^\alpha p(\X,t)-\Delta_m^{\beta/2,\lambda}p(\X,t)=f(\X,t)  &&\textrm{in}~~ \Omega,  \\
& D_t^{1-\alpha} \Delta_m^{\beta/2,\lambda}p(\X,t)=g(\X,t)   &&\textrm{in}~~\mathbb{R}^n \backslash \Omega, \\[5pt]
& p(\X,0)=p_0(\X)   &&\textrm{in}~~\Omega,
\end{aligned} \right.
\end{equation}
where $D_t^{1-\alpha}$ is the Riemann-Liouville derivative, defined in \eqref{RLderi}. It should be noted that the Neumann boundary condition $g(\X,t)$ is time dependent both in \eqref{IBFracLap2} and \eqref{IBFracLap3}, meaning that the numerical flux of diffusing particles across the boundary $\partial \Omega$ is time dependent.
\end{remark}

\begin{remark}
  For the problem \eqref{IBFracLap2} with homogeneous Neumann boundary conditions $g=0$ and source term $f=0$, if the PDF $m(\Y)$ is symmetric, we can prove the property of conservation of mass inside $\Omega$.

More specifically, from the symmetry of $m(\Y)$, we have
  {\small
  \begin{equation*}
   \begin{split}
   & \int\!\!\!\int_{\Omega \times \Omega} \frac{p(\X)-p(\Y)}{e^{\lambda|\X-\Y|}|\X-\Y|^{n+\beta}}m(\X-\Y)d\X d\Y
    \\
  &  =\int\!\!\!\int_{\Omega \times \Omega} \frac{p(\Y)-p(\X)}{e^{\lambda|\X-\Y|}|\X-\Y|^{n+\beta}}m(\X-\Y)d\X d\Y=0.
   \end{split}
  \end{equation*}
  }
  Therefore, for \eqref{IBFracLap2} with $f=g=0$,
  \begin{equation*}
    \begin{split}
      \frac{\partial}{\partial t} \int_\Omega p \,d\X
      &= \int_\Omega \Delta_m^{\beta/2,\lambda}p(\X) \, d\X  \\
      &= -\frac{1}{|\Gamma(-\beta)|}\,\int_\Omega\int_{\mathbb{R}^n} \frac{p(\X)-p(\Y)}{e^{\lambda|\X-\Y|}|\X-\Y|^{n+\beta}}m(\X-\Y)d\Y d\X  \\
      &= -\frac{1}{|\Gamma(-\beta)|}\,\int_\Omega\int_{\mathbb{R}^n\backslash\Omega} \frac{p(\X)-p(\Y)}{e^{\lambda|\X-\Y|}|\X-\Y|^{n+\beta}}m(\X-\Y)d\Y d\X  \\
      &= -\frac{1}{|\Gamma(-\beta)|}\,\int_{\mathbb{R}^n\backslash\Omega}\int_\Omega \frac{p(\X)-p(\Y)}{e^{\lambda|\X-\Y|}|\X-\Y|^{n+\beta}}m(\X-\Y)d\X d\Y  \\
      &= -\frac{1}{|\Gamma(-\beta)|}\,\int_{\mathbb{R}^n\backslash\Omega}\int_{\mathbb{R}^n} \frac{p(\X)-p(\Y)}{e^{\lambda|\X-\Y|}|\X-\Y|^{n+\beta}}m(\X-\Y)d\X d\Y  \\
      &= -\frac{1}{|\Gamma(-\beta)|}\,\int_{\mathbb{R}^n\backslash\Omega} \Delta_m^{\beta/2,\lambda}p(\Y) \,d\Y =0 .
    \end{split}
  \end{equation*}
  Thus, the quantity $\int_\Omega p d\X$ does not depend on $t$, which means the conservation of mass inside $\Omega$.
%  \begin{equation*}
%    \int\!\!\!\int_{\mathbb{R}^n\backslash\Omega \times \mathbb{R}^n\backslash\Omega} \frac{p(\X)-p(\Y)}{e^{\lambda|\X-\Y|}|\X-\Y|^{n+\beta}}m(\X-\Y)d\X d\Y
%    =\int\!\!\!\int_{\mathbb{R}^n\backslash\Omega \times \mathbb{R}^n\backslash\Omega} \frac{p(\Y)-p(\X)}{e^{\lambda|\X-\Y|}|\X-\Y|^{n+\beta}}m(\X-\Y)d\X d\Y
%  \end{equation*}
\end{remark}

Based on the definition of $\Delta_m^{\beta/2,\lambda}p(\X,t)$ in \eqref{anisoLapTempFinal}, there is no need for the solution $p(\X,t)$ to vanish at infinity. To guarantee the convergence of the integral in \eqref{anisoLapTempFinal}, the solution $p(\X,t)$ should satisfy that there exist positive $M$ and $C$ such that when $|\X|>M$,
\begin{equation*}
  |p(\X,t)|<C e^{(\lambda-\epsilon)|\X|}  \qquad \textrm{for positive small } \epsilon.
\end{equation*}
This is an essential difference from Riesz fractional derivatives \cite{Yang:10}, which must vanish at infinity. A special example is that $p(\X,t)\equiv1$ and $\Delta_m^{\beta/2,\lambda}1\equiv0$. Indeed, that $p(\X,t)$ does not vanish at infinity still has some clear physical meaning, e.g., escape probability \cite{DengWW:17}.
Considering the case of $\beta=2$ in \eqref{anisoLapTempFT}, we have
\begin{equation}\label{GBCbeta2}
  \mathscr{F}[\Delta_m^{1,\lambda}p(\X,t)]= \int_{|\fai|=1}  \Big[ -(\k\cdot\fai)^2-2\lambda(i\k\cdot\fai) \Big] m(\fai) d\fai\cdot \hat{p}(\k,t) .
\end{equation}
In this case, $m(\fai)$ determines the covariance matrix $\mathbf{a}$ in \eqref{LKformula} \cite{Meerschaert:99}. If $m(\fai)$ is symmetric, the term containing $i\k$, corresponding to the first order derivative, vanishes. If not, from \eqref{GBCbeta2},
\begin{equation}\label{IBbetaequal2}
\begin{split}
    \mathscr{F}[\Delta_m^{1,\lambda}p(\X,t)]
    = \left((i\k)^TA(i\k)-2\lambda(i\k)^T\mathbf{b}\right)\hat{p}(\k,t),
\end{split}
\end{equation}
where the matrix $A=(a_{ij})_{n\times n}$ with $a_{ij}=\int_{|\fai|=1}\fai_i\fai_j \,m(\fai)d\fai$
and the vector $\mathbf{b}=(b_j)_{n\times 1}$ with $b_j=\int_{|\fai|=1}\fai_j \,m(\fai)d\fai$. This implies
\begin{equation}\label{LaplacianD1}
  \Delta_m^{1,\lambda}=\sum_{i,j=1}^n a_{ij}\frac{\partial^2}{\partial\X_i \partial\X_j} + 2\lambda\sum_{j=1}^n b_j\frac{\partial}{\partial \X_j}.
\end{equation}
Then the weak solution $p\in H^{1}(\mathbb{R}^n)$ of \eqref{IBFracLap2} satisfies, for all $q\in H^{1}(\mathbb{R}^n)$,
\begin{equation*}
  \int_\Omega \frac{\partial p}{\partial t}q d\X + \int_{\mathbb{R}^n}\sum_{i,j=1}^n a_{ij} \frac{\partial p}{\partial\X_i} \frac{\partial q}{\partial\X_j} d\X
  -2\lambda\int_{\mathbb{R}^n}\sum_{j=1}^n b_j \frac{\partial p}{\partial\X_j} \,q d\X
  =\int_\Omega fq d\X -\int_{\mathbb{R}^n\backslash\Omega}gq d\X  .
\end{equation*}
For the Neumann boundary conditions in \eqref{IBFracLap2}, we have
\begin{equation*}
\begin{split}
    &\int_{\mathbb{R}^n\backslash\Omega}gq d\X
    = \int_{\mathbb{R}^n\backslash\Omega} \sum_{i,j=1}^n a_{ij}\frac{\partial^2 p}{\partial\X_i \partial\X_j}\,q d\X
       +2\lambda\int_{\mathbb{R}^n\backslash\Omega}\sum_{j=1}^n b_j \frac{\partial p}{\partial\X_j} \,q d\X     \\
    =& ~-\int_{\partial\Omega} \sum_{i,j=1}^n a_{ij} \frac{\partial p}{\partial \mathbf{n}_i} \,q ds
        -\int_{\mathbb{R}^n\backslash\Omega} \sum_{i,j=1}^n a_{ij}\frac{\partial p}{\partial\X_i} \frac{\partial q}{\partial\X_j} d\X
        +2\lambda\int_{\mathbb{R}^n\backslash\Omega}\sum_{j=1}^n b_j \frac{\partial p}{\partial\X_j} \,q d\X .
\end{split}
\end{equation*}
Then
\begin{equation*}
\begin{split}
 & \int_\Omega \frac{\partial p}{\partial t}q d\X + \int_{\Omega}\sum_{i,j=1}^n a_{ij} \frac{\partial p}{\partial\X_i} \frac{\partial q}{\partial\X_j} d\X
   -2\lambda\int_{\Omega}\sum_{j=1}^n b_j \frac{\partial p}{\partial\X_i} \,q d\X
  \\
 & =\int_\Omega fq d\X + \int_{\partial\Omega} \sum_{i,j=1}^n a_{ij} \frac{\partial p}{\partial \mathbf{n}_i} \,q ds,
\end{split}
\end{equation*}
which means that the usual Neumann boundary condition is recovered. Similarly, for the Dirichlet boundary condition in \eqref{IBFracLap1}, when $\beta=2$, $\Delta_m^{1,\lambda}$ becomes a local operator. Then only the information of $g(\X,t)$ on the boundary $\partial\Omega$ is used to solve the problem, implying that the usual Dirichlet boundary condition is recovered.

\section{Well-posedness and regularity}\label{Sec6}
Here we show the well-posedness of the problems provided in the above section. First we define the fractional Sobolev space for $s\in(0,1)$,
\begin{equation*}
  H^s(\Omega):= \left\{v\in L^2(\Omega): |v|_{H^s(\Omega)}<\infty\right\},
\end{equation*}
where
\begin{equation*}
  |v|_{H^s(\Omega)}=\left(\int\!\!\!\!\!\int_{\Omega\times\Omega}\frac{(v(x)-v(y))^2}{|x-y|^{n+2s}} dxdy\right)^{1/2}
\end{equation*}
is the Aronszajn-Slobodeckij seminorm. The space $H^s(\Omega)$ is a Banach space, endowed with the norm
\begin{equation*}
  \|v\|_{H^s(\Omega)}:=\Big(\|v\|^2_{L^2(\Omega)}+|v|^2_{H^s(\Omega)}\Big)^{1/2}.
\end{equation*}
Equivalently, the space $H^s(\Omega)$ can be regarded as the restriction to $\Omega$ of functions in $H^s(\mathbb{R}^n)$. We define $H^s_0(\Omega)$ as the closure of $C_0^\infty(\Omega)$ in $H^s(\Omega)$.
Consider the space
\begin{equation*}
  \tilde{H}^s_0(\Omega)=\left\{v\in H^s(\mathbb{R}^n):~v=0~\textrm{in}~\mathbb{R}^n \backslash \Omega\right\}
\end{equation*}
equipped with the $H^s(\mathbb{R}^n)$ norm. The dual space of $\tilde{H}^s_0(\Omega)$ is denoted by $H^{-s}(\Omega)$ or $\tilde{H}^s_0(\Omega)'$.

If $g\in L^2(0,T;H^{\beta/2}({\mathbb R}^n))\cap H^1(0,T;H^{-\beta/2}(\mathbb{R}^n))$ and $f\in L^2(0,T;H^{-\beta/2}(\Omega))$, then the weak formulation of \eqref{IBFracLap1} is to find $p=\tilde{p}+g$ such that $\tilde{p}\in L^2(0,T;\tilde{H}^{\beta/2}_0(\Omega))\cap H^1(0,T;H^{-\beta/2}(\Omega))\hookrightarrow C([0,T];L^2(\Omega))$ and
\begin{equation}\label{IBweak1}
  \int_0^T \!\!\!\int_{\Omega} \partial_t \tilde{p}\, q\, d\X dt  + \frac{1}{2|\Gamma(-\beta)|} \int_0^T a(\tilde{p},q) dt
  = \int_0^T \!\!\!\int_\Omega (f+ \Delta_m^{\beta/2,\lambda}g-\partial_t g)\,q \,d\X dt
\end{equation}
for all $q\in L^2(0,T;\tilde{H}^{\beta/2}_0(\Omega))$, where
\begin{equation}\label{IBBilinear1}
\begin{split}
    a(\tilde{p},q) &= 2|\Gamma(-\beta)| \left(-\Delta_m^{\beta/2,\lambda}\tilde{p},q\right) \\
    &= 2\int\!\!\!\!\!\int_{\mathbb{R}^n\times\mathbb{R}^n} \frac{(\tilde{p}(\X)-\tilde{p}(\Y))}{e^{\lambda|\X-\Y|}|\X-\Y|^{n+\beta}} q(\X)m(\X-\Y)d\X d\Y   \\
    &= 2\int\!\!\!\!\!\int_{\mathbb{R}^n\times\mathbb{R}^n} \frac{(\tilde{p}(\Y)-\tilde{p}(\X))}{e^{\lambda|\X-\Y|}|\X-\Y|^{n+\beta}} q(\Y)m(\Y-\X)d\X d\Y   \\
    &= \int\!\!\!\!\!\int_{\mathbb{R}^n\times\mathbb{R}^n} \frac{(\tilde{p}(\X)-\tilde{p}(\Y))(q(\X)m(\X-\Y)-q(\Y)m(\Y-\X))}{e^{\lambda|\X-\Y|}|\X-\Y|^{n+\beta}} d\X d\Y.
\end{split}
\end{equation}
To show the well-posedness of the weak formulation $\eqref{IBweak1}$, the main task is to prove the continuity and coercivity of bilinear form $a(\tilde{p},q)$, while $l(q):=\int_\Omega (f+ \Delta_m^{\beta/2,\lambda}g-\partial_t g)\,q \,d\X$ is a continuous linear functional on $L^2(0,T;\tilde{H}^{\beta/2}_0(\Omega))$ evidently. Here, the bilinear form $a(\tilde{p},q)$ is based on \eqref{anisoLapTempCase1}. For \eqref{anisoLapTempCase2}, the bilinear form becomes a little bit complex. But the well-posedness still is valid since we mainly prove it in Fourier space.
\begin{lemma}\label{D_continuity}
  The bilinear form $a(p,q)$ is continuous on $H^{\beta/2}(\mathbb{R}^n)\times H^{\beta/2}(\mathbb{R}^n)$.
\end{lemma}
\begin{proof}
We prove the continuity in the Fourier space. Using the Parseval equality and Theorem \ref{theo1}, we have
\begin{equation}\label{IBcontinuity1}
  \begin{split}
    a(p,q)=& 2|\Gamma(-\beta)| (\mathscr{F}[-\Delta_m^{\beta/2,\lambda}p],\mathscr{F}[q]) \\
    =& 2\Gamma(-\beta) \int_{\mathbb{R}^n}\!\!\!\int_{|\fai|=1} \Big( \lambda^\beta - (\lambda^2+(\k\cdot\fai)^2)^{\beta/2} e^{-i\beta\eta}\Big) ~m(\fai)d\fai ~\hat{p}(\k)\overline{\hat{q}(\k)}d\k, \\
  \end{split}
\end{equation}
where $\eta=\arctan\left(\frac{\k\cdot\fai}{\lambda}\right)$.
Then because of $(\lambda^2+|\k\cdot\fai|^2)^{\beta/2}\leq 2^{\beta/2}(\lambda^\beta+|\k\cdot\fai|^\beta)$,
\begin{equation}
  \begin{split}
    |a(p,q)|
    \leq&~  C\int_{\mathbb{R}^n}\!\!\!\int_{|\fai|=1} (1+|\k\cdot\fai|^\beta) m(\fai)d\fai ~|\hat{p}(\k)||\hat{q}(\k)|d\k   \\
    \leq&~  C\int_{\mathbb{R}^n}\!\!\!\int_{|\fai|=1} (1+|\k|^\beta) m(\fai)d\fai  ~|\hat{p}(\k)||\hat{q}(\k)|d\k   \\
    =&~ C\int_{\mathbb{R}^n}(1+|\k|^\beta) ~|\hat{p}(\k)||\hat{q}(\k)|d\k  \\
    \leq&~ C\Big(\int_{\mathbb{R}^n} (1+|\k|^\beta) |\hat{p}(\k)|^2 d\k\Big)^{1/2} \cdot \Big(\int_{\mathbb{R}^n} (1+|\k|^\beta) |\hat{q}(\k)|^2 d\k \Big)^{1/2}  \\
    =&~ C \|p\|_{H^{\beta/2}(\mathbb{R}^n)} \cdot \|q\|_{H^{\beta/2}(\mathbb{R}^n)},
  \end{split}
\end{equation}
which completes the proof.
\end{proof}

Before proving the coercivity of the bilinear form $a(q,q)$, we show a Lemma first. Because of the Parseval equality, there exists
\begin{equation}
\begin{split}
  a(q,q)=& 2|\Gamma(-\beta)|\left(\mathscr{F}[-\Delta_m^{\beta/2,\lambda}q],\mathscr{F}[q]\right) \\
    =& 2\Gamma(-\beta) \int_{\mathbb{R}^n}\!\!\!\int_{|\fai|=1} \Big( \lambda^\beta - (\lambda^2+(\k\cdot\fai)^2)^{\beta/2} e^{-i\beta\eta}\Big) ~m(\fai)d\fai ~|\hat{q}(\k)|^2d\k \\
    =& 2\Gamma(-\beta) \int_{\mathbb{R}^n} d(\k) ~|\hat{q}(\k)|^2d\k,
\end{split}
\end{equation}
where $\eta=\arctan(\frac{\k\cdot\fai}{\lambda})$ and $d(\k)=\int_{|\fai|=1} \Big( \lambda^\beta - (\lambda^2+(\k\cdot\fai)^2)^{\beta/2} e^{-i\beta\eta}\Big) ~m(\fai)d\fai$. Thus the complex conjugate of $d(\k)$ satisfies $\overline{d(\k)}=d(-\k)$, which implies that $\Im[d(\k)]$ is an odd function. On the other hand, since $\hat{q}(\k)=\int_{\mathbb{R}^n}e^{i\k\cdot\X}q(\X)d\X$ and $q(\X)$ is a real function, we have $\hat{q}(\k)=\overline{\hat{q}(-\k)}$ and $|\hat{q}(\k)|^2$ is an even function by
\begin{equation*}
  |\hat{q}(\k)|^2=\hat{q}(\k)\overline{\hat{q}(\k)}=\overline{\hat{q}(-\k)}\hat{q}(-\k)=|\hat{q}(-\k)|^2.
\end{equation*}
Therefore, $\Im[a(q,q)]=0$ and
\begin{equation}
  a(q,q)= 2\Gamma(-\beta) \int_{\mathbb{R}^n}\!\!\!\int_{|\fai|=1}
  \Big( \lambda^\beta - (\lambda^2+(\k\cdot\fai)^2)^{\beta/2} \cos(\beta\eta)\Big) ~m(\fai)d\fai ~|\hat{q}(\k)|^2d\k. \\
\end{equation}
For the isotropic case, $m(\fai)$ is a constant, and
\begin{equation}\label{fstepFTiso}
    \mathscr{F}[-\Delta^{\beta/2,\lambda}q(\X)]
    = \frac{(-1)^{\lceil\beta\rceil}}{\omega_n} \int_{|\fai|=1} \Big( \lambda^\beta -  (\lambda^2+(\k\cdot\fai)^2)^{\frac{\beta}{2}} \cos(\beta\eta) \Big) ~d\fai \cdot \hat{q}(\k).
\end{equation}
In the following, we show that under some reasonable assumptions on $m(\fai)$ there exists a constant $C>0$ such that
\begin{equation}\label{fstepcompare}
 \Re[\mathscr{F}[-\Delta_m^{\beta/2,\lambda}q(\X)]\cdot \overline{\hat{q}(\k)}] \geq C\mathscr{F}[-\Delta^{\beta/2,\lambda}q(\X)]\cdot \overline{\hat{q}(\k)} \qquad \forall \,\k\in\mathbb{R}^n,
\end{equation}
where
\begin{equation*}
  \Re[\mathscr{F}[-\Delta_m^{\beta/2,\lambda}q(\X)] \cdot\overline{\hat{q}(\k)}]= (-1)^{\lceil\beta\rceil} \int_{|\fai|=1} \Big( \lambda^\beta -   (\lambda^2+(\k\cdot\fai)^2)^{\frac{\beta}{2}} \cos(\beta\eta) \Big)~m(\fai)d\fai\cdot |\hat{q}(\k)|^2.
\end{equation*}

\begin{definition}
	  A probability density function $m(\fai)$ on the unit sphere in $\mathbb{R}^n$ is said to be nondegenerate if the set $A_m(\fai):=\{\fai;m(\fai)\neq0\}$ can span the whole space $\mathbb{R}^n$.
\end{definition}

\begin{lemma}\label{ConditionForM}
  Let $\beta\in(0,1)\cup(1,2)$. For the operator $-\Delta_m^{\beta/2,\lambda}$, the non-degeneration of a probability density function $m(\fai)$ on the unit sphere is  equivalent to \eqref{fstepcompare}.
\end{lemma}
\begin{proof}
Denote $f(\k\cdot\fai)=(-1)^{\lceil\beta\rceil} (\lambda^\beta - (\lambda^2+(\k\cdot\fai)^2)^{\frac{\beta}{2}} \cos(\beta\eta) )$. Then $f'\geq0$ and $f_{\min}=f(0)=0$ \cite[Appendix]{Zhang:17}, which implies that  $\mathscr{F}[-\Delta^{\beta/2,\lambda}q(\X)]  \cdot \overline{\hat{q}(k)}\geq0$ and $\Re[\mathscr{F}[-\Delta_m^{\beta/2,\lambda}q(\X)] \cdot \overline{\hat{q}(k)}] \geq0$.
  If $\k=\mathbf{0}$, then \eqref{fstepcompare} holds. If $\k\neq\mathbf{0}$, then \eqref{fstepcompare} is equivalent to
  \begin{equation*}
    \frac{ \Re[\mathscr{F}[-\Delta_m^{\beta/2,\lambda}q(\X)]\cdot \overline{\hat{q}(k)}] } {\mathscr{F}[-\Delta^{\beta/2,\lambda}q(\X)]  \cdot \overline{\hat{q}(k)}}
    \geq C>0 \qquad \forall \,\k\in\mathbb{R}^n.
  \end{equation*}
First we prove the sufficiency. If the probability density function $m(\fai)$ is degenerate, i.e., $\textrm{span}\{A_m(\fai)\}$ is the strict subspace of $\mathbb{R}^n$, then there exists $\mathbb{Q}$ being the orthogonal complement of $\textrm{span}\{A_m(\fai)\}$ in $\mathbb{R}^n$, satisfying $\forall \k\in\mathbb{Q}$ and $\forall \fai\in A_m(\fai)$, $(\k\cdot\fai)=0$. In this case, there exist $\k,\fai\in\mathbb{Q}\subset\mathbb{R}^n$ s.t. $(\k\cdot\fai)>0$. It means that $\mathscr{F}[-\Delta^{\beta/2,\lambda}q(\X)]  \cdot \overline{\hat{q}(k)}>0$ but $\Re[\mathscr{F}[-\Delta_m^{\beta/2,\lambda}q(\X)] \cdot \overline{\hat{q}(k)}] =0$. Then \eqref{fstepcompare} does not hold.

  On the contrary, for necessity, we assume that $m(\fai)$ is nondegenerate. If $q(\X)$ does not equal to zero but $ \Re[\mathscr{F}[-\Delta_m^{\beta/2,\lambda}q(\X)]\cdot \overline{\hat{q}(k)}] =0$, then for any $\fai$ and $\k$, $f(\k\cdot\fai)m(\fai)=0$ almost everywhere. Since $m(\fai)$ is nondegenerate, $\k$ must be orthogonal to the space $\textrm{span}\{A_m(\fai)\}(=\mathbb{R}^n)$. So $\k$ must be a zero vector, which means that $ \Re[\mathscr{F}[-\Delta_m^{\beta/2,\lambda}q(\X)]\cdot \overline{\hat{q}(k)}] =0$  has the only zero point $\k=\mathbf{0}$ if $q(\X)$ is not zero. By a simple calculation, both $ \Re[\mathscr{F}[-\Delta_m^{\beta/2,\lambda}q(\X)]\cdot \overline{\hat{q}(k)}] $ and $\mathscr{F}[-\Delta^{\beta/2,\lambda}q(\X)]  \cdot \overline{\hat{q}(k)}$ are $\mathcal{O}(|\k|^2)$ when $|\k|\rightarrow0$ and $\mathcal{O}(|\k|^\beta)$ when $|\k|\rightarrow\infty$. Then Eq. \eqref{fstepcompare} holds.
\end{proof}

\begin{lemma}
  Let $q\in\tilde{H}^{\beta/2}_0(\Omega)$. If the probability density function $m(\fai)$ is nondegenerate, then the bilinear form $a(q,q)\geq C|q|^2_{H^{\beta/2}(\mathbb{R}^n)}$, i.e., it is coercive in $H^{\beta/2}(\mathbb{R}^n)$.
\end{lemma}
\begin{proof}
The coercivity is proved in two steps. The first step is to show that $a(q,q)$ can bound the bilinear form $\tilde{a}(q,q)$ with isotropic $m(\fai)$, i.e., $a(q,q)\geq C\tilde{a}(q,q)$, where
\begin{equation}
  \tilde{a}(p,q)= \int\!\!\!\!\!\int_{\mathbb{R}^n\times\mathbb{R}^n} \frac{(p(\X)-p(\Y))(q(\X)-q(\Y))}{e^{\lambda|\X-\Y|}|\X-\Y|^{n+\beta}}d\X d\Y.
\end{equation}
In the second step, we prove that $\tilde{a}(q,q)$ can be bounded by the norm $\|q\|^2_{H^{\beta/2}(\mathbb{R}^n)}$.

In the first step, we prove it in the Fourier space like Lemma \ref{D_continuity}.
It suffices to prove that there exists a positive constant $C$ such that
\begin{equation}\label{fstepcompare2}
\Re[\mathscr{F}[-\Delta_m^{\beta/2,\lambda}q(\X)]\cdot \overline{\hat{q}(k)}]  \geq C\mathscr{F}[-\Delta^{\beta/2,\lambda}q(\X)] \cdot \overline{\hat{q}(k)} \qquad \forall \,\k\in\mathbb{R}^n,
\end{equation}
which can be guaranteed if  $m(\fai)$ is nondegenerate from Lemma \ref{ConditionForM}.
See \cite[(5.11)]{Ervin:07} for some specific expressions of $m(\fai)$, where the two dimensional case is discussed, but without tempering.
In the second step, we adopt the technique of \cite[Proposition 3.2]{Zhang:17} by taking a sufficiently big ball $B_\rho$ centering at the origin with radius $\rho$, such that $\Omega\subset B_\rho$. Denote $\delta>0$ as the distance between $\Omega$ and $\partial B_\rho$, $\delta=\underset{\X\in\Omega,\Y\in\partial B_\rho}{\inf}|\X-\Y|$. Then for $q\in \tilde{H}_0^{\beta/2}(\Omega)$,
\begin{equation}
  \begin{split}
    |q|^2_{H^{\beta/2}(B_\rho)}
    = &~ \int_{B_\rho}\int_{B_\rho} \frac{(q(\X)-q(\Y))^2}{|\X-\Y|^{n+\beta}}d\X d\Y \\
    \geq &~ \int_\Omega q^2(\X) \int_{B_\rho\backslash\Omega} \frac{1}{|\X-\Y|^{n+\beta}} d\Y d\X  \\
    \geq &~ (2\rho)^{-n-\beta}|B_\rho\backslash\Omega| \,\int_\Omega q^2(\X) d\X \\
    = &~ C \|q\|^2_{L^2(\Omega)}=C \|q\|^2_{L^2(\mathbb{R}^n)}
  \end{split}
\end{equation}
and
\begin{equation}
\begin{split}
    |q|^2_{H^{\beta/2}(\mathbb{R}^n)}
    = &~ |q|^2_{H^{\beta/2}(B_\rho)} + 2 \int_\Omega \int_{\mathbb{R}^n\backslash B_\rho} \frac{(q(\X))^2}{|\X-\Y|^{n+\beta}} d\Y d\X \\
    \leq & ~|q|^2_{H^{\beta/2}(B_\rho)} + 2 \int_\Omega (q(\X))^2 d\X \int_{\mathbb{R}^n\backslash B_\delta} |\Y|^{-n-\beta} d\Y  \\
    = & ~|q|^2_{H^{\beta/2}(B_\rho)} + \frac{2\omega_n \delta^{-\beta}}{\beta} \|q\|^2_{L^2(\Omega)}  \\
    \leq & ~C|q|^2_{H^{\beta/2}(B_\rho)}.
\end{split}
\end{equation}
Therefore,
\begin{equation}
  \begin{split}
    \tilde{a}(q,q)=& \int\!\!\!\!\!\int_{\mathbb{R}^n\times\mathbb{R}^n} \frac{(q(\X)-q(\Y))^2}{e^{\lambda|\X-\Y|}|\X-\Y|^{n+\beta}}d\X d\Y  \\
    \geq & \int\!\!\!\!\!\int_{B_\rho\times B_\rho} \frac{(q(\X)-q(\Y))^2}{e^{\lambda|\X-\Y|}|\X-\Y|^{n+\beta}}d\X d\Y  \\
    \geq & ~e^{-2\lambda\rho} |q|^2_{H^{\beta/2}(B_\rho)},  \\
    \geq & ~C \|q\|^2_{H^{\beta/2}(\mathbb{R}^n)}.
  \end{split}
\end{equation}
The proof is completed.
\end{proof}

\begin{theorem}[Existence and uniqueness of weak solutions]
Let $p_0\in L^2(\Omega)$, $f\in L^2(0,T;H^{-\beta/2}(\Omega))$ and $g\in L^2(0,T;H^{\beta/2}({\mathbb R}^n))\cap H^1(0,T;H^{-\beta/2}(\mathbb{R}^n))$. If the probability density function $m(\fai)$ is nondegenerate, there exists a unique weak solution of \eqref{IBFracLap1} in the sense of \eqref{IBweak1}.
\end{theorem}
\begin{proof}
  The continuity and coercivity of bilinear form $a(\tilde{p},q)$ have been obtained. Furthermore, $l(q)$ is a continuous linear functional. Then
  %  Combining $l(q) $ being a continuous linear functional,
%  by \cite[Theorem 30.A]{Zeidler:90},
  the original initial boundary value problem \eqref{IBFracLap1} has a unique solution.
\end{proof}

For the Neumann problem \eqref{IBFracLap2}, firstly we define the tempered fractional space
\begin{equation*}
H^{\beta/2,\lambda}(\mathbb{R}^n)= \left\{v\in L^2(\mathbb{R}^n):
|v|_{H^{\beta/2,\lambda}(\mathbb{R}^n)} <\infty \right\},
\end{equation*}
where the seminorm
\begin{equation}
|v|_{H^{\beta/2,\lambda}(\mathbb{R}^n)}= \left(\int_{\mathbb{R}^n}\int_{\mathbb{R}^n} \frac{(v(\X)-v(\Y))^2}{e^{\lambda|\X-\Y|}|\X-\Y|^{n+\beta}} d\X d\Y\right)^{1/2}
\end{equation}
and the norm
\begin{equation}
\|v\|_{H^{\beta/2,\lambda}(\mathbb{R}^n)}= \left( \|v\|_{L^2(\mathbb{R}^n)}^2+|v|^2_{H^{\beta/2,\lambda}(\mathbb{R}^n)} \right)^{1/2}.
\end{equation}

The main difference of Neumann problem with Dirichlet one is that essentially it is an unbounded problem. There are also some interesting properties for the operator $\Delta_m^{\beta/2,\lambda}$ defined in unbounded domain, e.g., $\Delta_m^{\beta/2,\lambda}1=0$ for constant 1,
%the exact solution may not vanish at the infinity.
%One of the most typical examples is constant 1 with $\Delta_m^{\beta/2,\lambda}1=0$.
which may produce some dedicated/complicated issues for the choice of function spaces, ways of proving the well-posedness, etc. For example, for the bilinear form $a(\cdot,\cdot)$ in \eqref{IBBilinear1},
%does not have the continuity, i.e.,
\begin{equation}\label{Counterexample}
  |a(p,q)| \nleqslant  C\, |p|_{H^{\beta/2,\lambda}(\mathbb{R}^n)} \cdot |q|_{H^{\beta/2,\lambda}(\mathbb{R}^n)}.
\end{equation}
In fact, take $n=1$, $q(x)\equiv1$ and
\begin{equation*}
  p(x)=\left\{
  \begin{array}{cc}
  -1 & x<0  \\ 0 & x\geq0,
  \end{array}\right.
  \qquad
  m(x)=\left\{
  \begin{array}{cc}
  0 & x=-1  \\ 1 & x=1.
  \end{array}\right.
\end{equation*}
Then the right hand side of \eqref{Counterexample} equals to $0$ while the left hand side
\begin{equation*}
\begin{split}
   a(p,q)&=2\int_{-\infty}^{\infty}\int_{-\infty}^{\infty} \frac{p(x)-p(y)}{e^{\lambda|x-y|}|x-y|^{1+\beta}} m(\textrm{sgn}(x-y)) dx dy  \\
   &=2\int_{-\infty}^{\infty}\int_{y}^{\infty} \frac{p(x)-p(y)}{e^{\lambda|x-y|}|x-y|^{1+\beta}} dx dy \\
   &=2\int_{-\infty}^{0}\int_{0}^{\infty} \frac{1}{e^{\lambda|x-y|}|x-y|^{1+\beta}} dx dy >0.
\end{split}
\end{equation*}

In the following, we just focus on the case that the probability density function $m(\Y)$ is symmetric.
We define the function space, containing the functions that may not vanish at the infinity,
\begin{equation}
\mathbb{V}=\{ p\in L^2(\Omega):|p|_{H_m^{\beta/2,\lambda}(\mathbb{R}^n)}<\infty\},
\end{equation}
furnished with the norm
\begin{equation}\label{Vnorm}
\begin{split}
\|p\|_{\mathbb{V}}= \left(\|p\|_{L^2(\Omega)}^2+|p|_{H_m^{\beta/2,\lambda}(\mathbb{R}^n)}^2\right)^{1/2},  \qquad\qquad\qquad\qquad  \\
|p|_{H_m^{\beta/2,\lambda}(\mathbb{R}^n)}=\left(\int_{\mathbb{R}^n}\int_{\mathbb{R}^n} \frac{(p(\X)-p(\Y))^2}{e^{\lambda|\X-\Y|}|\X-\Y|^{n+\beta}}
m(\X-\Y)d\X d\Y\right)^{1/2}.
\end{split}
\end{equation}

\begin{proposition}
  $\mathbb{V}$ is a Hilbert space with the norm defined in \eqref{Vnorm}.
\end{proposition}
\begin{proof}
We first verify that the norm in \eqref{Vnorm} is well-defined. Let $\|p\|_{\mathbb{V}}=0$. It can be easily obtained that $p=0$ a.e. in $\Omega$ from $\|p\|_{L^2(\Omega)}=0$. Then from $|p|_{H_m^{\beta/2,\lambda}(\mathbb{R}^n)}=0$, one gets that  $(p(\X)-p(\Y))^2m(\X-\Y)=0$ a.e. in $\mathbb{R}^n$. Since $m$ is nondegenerate, there exist $n$ linearly independent nonzero vectors $\mathbf{r}_i,i=1,\cdots,n,$ satisfying $m(\mathbf{r}_i)\neq0$. Thus $\mathbf{r}_i\cdot\nabla p=0$, which implies that $\nabla p$ is orthogonal to $\mathbf{r}_i$ for all $i=1,\cdots,n$. Therefore, $\nabla p=\mathbf{0}$, which leads to $p=0$ a.e. in $\mathbb{R}^n$, by combining with $p=0$ a.e. in $\Omega$.

%. Combined with $p=0$ a.e. in $\Omega$, we obtain $p=0$ a.e. in $\mathbb{R}^n$.

Then we prove that $\mathbb{V}$ is complete by imitating the proof of \cite[Proposition 3.1]{Dipierro:14}. Take a Cauchy sequence $p_k$ with respect to the norm in \eqref{Vnorm}. In particular, $p_k$ is a Cauchy sequence in $L^2(\Omega)$ and therefore, up to a subsequence, we suppose that $p_k$ converges to some $p$ in $L^2(\Omega)$ and a.e. in $\Omega$.
On the other hand, for any $(\X,\Y)\in \mathbb{R}^{2n}$, define
\begin{equation}\label{EU}
  E_{p_k}(\X,\Y) := (p_k(\X)-p_k(\Y)) \frac{m^{1/2}(\X-\Y)}{e^{\lambda|\X-\Y|/2}|\X-\Y|^{(n+\beta)/2}}.
\end{equation}
Accordingly, since $p_k$ is a Cauchy sequence in $\mathbb{V}$ , for any $\varepsilon>0$, there exists $N_\varepsilon \in \mathbb{N}$ such that if $k,k' \geq N_\varepsilon$ then
$$ \varepsilon^2 \geq \int_{\mathbb{R}^{2n}} |(p_k-p_{k'})(\X)-(p_k -p_{k'})(\Y)|^2 \frac{m(\X-\Y)}{e^{\lambda|\X-\Y|}|\X-\Y|^{n+\beta}} d\X d\Y = \|E_{p_k}- E_{p_k'}\|^2_{L^2(\mathbb{R}^{2n})},$$
which means that $E_{p_k}$ is a Cauchy sequence in $L^2(\mathbb{R}^{2n})$. Up to a subsequence, we assume
that $E_{p_k}$ converges to some $E$ in $L^2(\mathbb{R}^{2n})$ and a.e. in $\mathbb{R}^{2n}$.

Fixing $\X_0\in \Omega$, there exists $\lim_{k\rightarrow\infty}  p_k(\X_0) = p(\X_0)$; then for any given $\Y\in\mathbb{R}^n\backslash\Omega$, we have that
$$
\lim_{k\rightarrow\infty} E_{p_k}(\X_0,\Y) = E(\X_0,\Y).$$
Noticing that
\begin{equation*}
  E_{p_k}(\X_0,\Y):=(p_k(\X_0)-p_k(\Y)) \frac{m^{1/2}(\X_0-\Y)}{e^{\lambda|\X_0-\Y|/2}|\X_0-\Y|^{(n+\beta)/2}},
\end{equation*}
there exists
\begin{equation*}
  \begin{split}
    \lim_{k\rightarrow\infty} p_k(\Y)
    &=\lim_{k\rightarrow\infty} \left(p_k(\X_0)- \frac{e^{\lambda|\X_0-\Y|/2}|\X_0-\Y|^{(n+\beta)/2}}{m^{1/2}(\X_0-\Y)} E_{p_k}(\X_0,\Y) \right)  \\
    &= p(\X_0)- \frac{e^{\lambda|\X_0-\Y|/2}|\X_0-\Y|^{(n+\beta)/2}}{m^{1/2}(\X_0-\Y)} E(\X_0,\Y)
  \end{split}
\end{equation*}
for a.e. $\Y\in \mathbb{R}^n\backslash\Omega$.
This means that $p_k$ converges to some $p$ a.e. in $\mathbb{R}^n$. So, using that $p_k$ is a Cauchy sequence in $\mathbb{V}$, fixed any $\varepsilon>0$, there exists $N_\varepsilon\in\mathbb{N}$ such that, for any $k'\geq N_\varepsilon$,
\begin{equation*}
  \begin{split}
    \varepsilon^2
    &\geq \underset{k\rightarrow\infty}{\liminf} \|p_k-p_{k'}\|^2_\mathbb{V}  \\
    &= \underset{k\rightarrow\infty}{\liminf} \int_\Omega (p_k-p_{k'})^2d\X  \\
        &~~~ +\underset{k\rightarrow\infty}{\liminf}\int_{\mathbb{R}^n}\!\!\!\int_{\mathbb{R}^n}|(p_k-p_{k'})(\X)-(p_k-p_{k'})(\Y)|^2 \frac{m(\X-\Y)}{e^{\lambda|\X-\Y|}|\X-\Y|^{n+\beta}} d\X d\Y  \\
    &\geq \int_\Omega (p-p_k')^2d\X+
         \int_{\mathbb{R}^n}\!\!\!\int_{\mathbb{R}^n}((p-p_{k'})(\X)-(p-p_{k'})(\Y))^2 \frac{m(\X-\Y)}{e^{\lambda|\X-\Y|}|\X-\Y|^{n+\beta}} d\X d\Y \\
    &= \|p-p_{k'}\|^2_{\mathbb{V}},
  \end{split}
\end{equation*}
where the Fatous Lemma is used. This says that $p_k'$ converges to $p$ in $\mathbb{V}$, showing that $\mathbb{V}$ is complete.
\end{proof}

Then the weak formulation of \eqref{IBFracLap2} is to find $p\in L^2(0,T;\mathbb{V})\cap H^1(0,T;\mathbb{V}')$ satisfying
\begin{equation}\label{IBweak2}
 \int_{0}^{T}\!\!\!\int_\Omega \frac{\partial p}{\partial t}\,q d\X dt  +  \frac{1}{2|\Gamma(-\beta)|} \int_{0}^{T} a(p,q) dt
= \int_{0}^{T}\!\!\!\int_{\Omega} f\,q d\X dt - \int_{0}^{T}\!\!\!\int_{{\mathbb R}^n\backslash\Omega}gq \,d\X dt
\end{equation}
for all $q\in L^2(0,T;\mathbb{V})$, where
\begin{equation}
  a(p,q)=\int\!\!\!\!\!\int_{\mathbb{R}^n\times\mathbb{R}^n} \frac{(p(\X)-p(\Y))(q(\X)-q(\Y))}{e^{\lambda|\X-\Y|}|\X-\Y|^{n+\beta}} m(\X-\Y)d\X d\Y.
\end{equation}
Similar to \cite[Theorem 4.2]{Deng:17}, we have
\begin{theorem}[Existence and uniqueness of weak solutions]
Let $p_0\in L^2(\Omega)$, $f\in L^2(0,T;L^2(\Omega))$ and $g\in L^2(0,T;\mathbb{V}')$. If $m(\Y)$ is nondegenerate, then there exists a unique weak solution of \eqref{IBFracLap2} in the sense of \eqref{IBweak2}.
\end{theorem}
\begin{proof}
Let $t_k=k\tau$, $k=0,1,\dots,N$, be a partition of the time interval $[0,T]$, with step size $\tau=T/N$, and define
\begin{equation*}
f_k(\X):=\frac{1}{\tau}\int_{t_{k-1}}^{t_k} f(\X,t)dt,  \quad
g_k(\X):=\frac{1}{\tau}\int_{t_{k-1}}^{t_k} g(\X,t)d t ,\quad k=1,\dots,N.
\end{equation*}
Then consider the time discrete problem: for a given $p_{k-1}\in\mathbb{V}$, find $p_k\in\mathbb{V}$ such that
\begin{align}\label{TD-PDE}
& \frac{1}{\tau} \int_\Omega p_k(\X)q(\X)d\X+ \frac{1}{2|\Gamma(-\beta)|}\, a(p_k(\X),q(\X)) \nonumber\\
&=  \frac{1}{\tau} \int_\Omega p_{k-1}(\X)q(\X)d\X + \int_{\Omega} f_k({\bf X})q(\X)d\X -\int_{{\mathbb R}^n\backslash\Omega}g_k(\X)q(\X)d\X
\quad\forall\, q\in \mathbb{V} .
\end{align}
From the definition of $\mathbb{V}$ in \eqref{Vnorm}, the continuity and coercivity of $a(p,q)$ of the left hand side of \eqref{TD-PDE} on $\mathbb{V}$ is evident.
For the last term on the right hand side, we define $g(\X)=0,\X\in\Omega$ for supplementary. Then $g_k(\X)=0,\X\in\Omega$, and
\begin{align*}
  \left|\int_{{\mathbb R}^n\backslash\Omega}g_k(\X)q(\X)d\X \right|
  =\left| \int_{{\mathbb R}^n}g_k(\X)q(\X)d\X  \right|
  \leq \|g_k(\X)\|_{\mathbb{V}'} \|q(\X)\|_{\mathbb{V}}.
\end{align*}

Thus, the right hand side of \eqref{TD-PDE} satisfies
\begin{align*}
    \textrm{RHS}\leq &~C\|p_{k-1}\|_{L^2(\Omega)}\cdot\|q\|_{L^2(\Omega)}+\|f_k\|_{L^2(\Omega)}\cdot\|q\|_{L^2(\Omega)}
                       +\|g_k\|_{\mathbb{V}'} \|q\|_{\mathbb{V}} \\
    \leq &~C\Big(\|p_{k-1}\|_{L^2(\Omega)}+\|f_k\|_{L^2(\Omega)}+\|g_k\|_{\mathbb{V}'}\Big)\cdot \|q\|_{\mathbb{V}},
\end{align*}
which implies that the right hand side is a continuous linear functional on $\mathbb{V}$. Therefore, by the Lax-Milgram Lemma, there exists a unique solution $p_k\in\mathbb{V}$ for \eqref{TD-PDE}. Then using the technique in \cite[Theorem 4.2]{Deng:17}, there exists a unique solution $p$ satisfying \eqref{IBweak2}.
\end{proof}

\section{Conclusion}
This is a companion paper with the latest one  \cite{Deng:17}. The main generalizations come from three aspects: 1. the diffusion operators characterizing (normal or anomalous) diffusion without scaling limit are presented; 2. very general anisotropic diffusion operators describing nonhomogeneous  phenomena are proposed; 3. the tempered anisotropic diffusion operators are introduced by two different ways with different motivations, and they are proved to be equivalent; 4. the well-posedness and regularity of the anisotropic diffusion equations are discussed; 5. the models for the anisotropic anomalous diffusion with multiple internal states are built, including the Fokker-Planck and Feymann-Kac equations, respectively, governing the PDF of positions of particles and the PDF of the functional of the particles' trajectories. More wide applications and numerical methods for the newly built various models will be detailedly discussed in the near future.

%This paper is the continuation of the last one \cite{Deng:17}. Here, we discuss more general models, especially the nonlocal normal diffusion without scaling limit and the anisotropic anomalous diffusion. For anomalous diffusion, different ways of particles' movements lead to different macroscopic equations. Even for normal diffusion, without scaling limit, this difference also exists. We derive the anisotropic (tempered) fractional Laplacian $\Delta_m^{\beta/2}$ and $\Delta_m^{\beta/2,\lambda}$ from compound Poisson process and prove the equivalence with an alternative definition in Fourier space. We also give the Fokker-Planck and Feymann-Kac equations with multiple internal states with our new defined operators. For the fractional PDEs with the anisotropic tempered fractional Laplacian $\Delta_m^{\beta/2,\lambda}$, we prove their well-posedness by require the probability density function $m(\Y)$ being nondegenerate. More wide applications and numerical algorithms for the fractional PDEs with $\Delta_m^{\beta/2,\lambda}$ will be considered in future work.

\section*{Acknowledgements}
We thank Mark M. Meerschaert for fruitful discussions, especially another motivation of defining tempered fractional operators (See Eq. (\ref{anisoLapTempFT})).

%\section{Appendix}
\appendix
\section{Proof of Theorem \ref{theo1}}

\begin{proof}
  We mainly prove the equivalence of the anisotropic tempered fractional Laplacian in \eqref{anisoLapTempCase2} of Case II to the alternative definition \eqref{anisoLapTempFT}. The equivalence of the anisotropic fractional Laplacian of Case I and definition \eqref{anisoLapFT} can be obtained similarly. Taking Fourier transform of the right hand side of \eqref{anisoLapTempCase2} leads to
  \begin{equation*}
  \begin{split}
     &\mathscr{F}\Big[\Delta_m^{\beta/2,\lambda} p(\X,t)\Big](\k)  \\
     &= \frac{1}{|\Gamma(-\beta)|} \int_{\mathbb{R}^n}\frac{e^{i\k\cdot\Y}-1-i\k\cdot\Y}{e^{\lambda|\Y|}|\Y|^{n+\beta}}m(\Y)d\Y \cdot\hat{p}(\k,t)
     \\
     &
    ~~~   - \frac{1}{|\Gamma(-\beta)|} \Gamma(1-\beta)\lambda^{\beta-1} (-i\k\cdot\mathbf{b})\hat{p}(\k,t) \\
     &= \frac{1}{|\Gamma(-\beta)|} \Bigg[\int_{\mathbb{R}^n}\frac{\cos(\k\cdot\Y)-1}{e^{\lambda|\Y|}|\Y|^{n+\beta}}m(\Y)d\Y
           +i\int_{\mathbb{R}^n}\frac{\sin(\k\cdot\Y)-\k\cdot\Y}{e^{\lambda|\Y|}|\Y|^{n+\beta}}m(\Y)d\Y \Bigg]  \cdot\hat{p}(\k,t) \\
      &~~~+ \frac{1}{|\Gamma(-\beta)|} \Gamma(1-\beta)\lambda^{\beta-1} \int_{|\fai|=1} (i\k\cdot\fai) m(\fai)d\fai \cdot\hat{p}(\k,t).
  \end{split}
  \end{equation*}
  Since $\beta\in(1,2)$ in this case, by polar coordinate transformation and integration by parts, we have,
  \begin{equation*}
    \begin{split}
      \int_{\mathbb{R}^n}&\frac{1-\cos(\k\cdot\Y)}{e^{\lambda|\Y|}|\Y|^{n+\beta}}m(\Y)d\Y   \\
      =&\int_0^\infty\int_{|\fai|=1} r^{-1-\beta}e^{-\lambda r}(1-\cos(r\k\cdot\fai))m(\fai)d\fai dr \\
      =&\frac{\lambda^2}{(-\beta)(1-\beta)}\int_0^\infty r^{1-\beta}e^{-\lambda r} \int_{|\fai|=1}(1-\cos(r\k\cdot\fai))m(\fai)d\fai dr \\
      &-\frac{2\lambda}{(-\beta)(1-\beta)}\int_0^\infty r^{1-\beta}e^{-\lambda r} \int_{|\fai|=1}(\k\cdot\fai) \sin(r\k\cdot\fai) m(\fai)d\fai dr \\
      &+\frac{1}{(-\beta)(1-\beta)}\int_0^\infty r^{1-\beta}e^{-\lambda r} \int_{|\fai|=1} (\k\cdot\fai)^2 \cos(r\k\cdot\fai) m(\fai)d\fai dr \\
      =& \int_{|\fai|=1} (I_1+I_2+I_3) ~m(\fai)d\fai.
    \end{split}
  \end{equation*}
  Then using the formulae \cite[Eq. (3.944(5-6))]{Gradshteyn:80} and taking $\eta=\arctan\frac{(\k\cdot\fai)}{\lambda}$, we get
  \begin{equation*}
  \begin{split}
    &I_1=\Gamma(-\beta)\lambda^\beta-\Gamma(-\beta)(\lambda^2+(\k\cdot\fai)^2)^{\frac{\beta}{2}-1} \cdot \lambda^2\cos((2-\beta)\eta),\\
    &I_2=-2\Gamma(-\beta)(\lambda^2+(\k\cdot\fai)^2)^{\frac{\beta}{2}-1} \cdot (\k\cdot\fai)\lambda\sin((2-\beta)\eta),\\
    &I_3=\Gamma(-\beta)(\lambda^2+(\k\cdot\fai)^2)^{\frac{\beta}{2}-1} \cdot (\k\cdot\fai)^2\cos((2-\beta)\eta),
  \end{split}
  \end{equation*}
which results in
  \begin{equation*}
      I_1+I_2+I_3=\Gamma(-\beta) \Big(\lambda^\beta-(\lambda^2+(\k\cdot\fai)^2)^{\frac{\beta}{2}} \cos(\beta\eta) \Big).
  \end{equation*}
Then
  \begin{equation}\label{equiv1}
    \int_{\mathbb{R}^n}\frac{\cos(\k\cdot\Y)-1}{e^{\lambda|\Y|}|\Y|^{n+\beta}}m(\Y)d\Y
    = -\Gamma(-\beta) \int_{|\fai|=1} \Big(\lambda^\beta-(\lambda^2+(\k\cdot\fai)^2)^{\frac{\beta}{2}} \cos(\beta\eta)\Big) ~m(\fai)d\fai.
  \end{equation}
  Similarly,
  \begin{equation}\label{equiv2}
    \begin{split}
        \int_{\mathbb{R}^n}&\frac{\sin(\k\cdot\Y)-\k\cdot\Y}{e^{\lambda|\Y|}|\Y|^{n+\beta}}m(\Y)d\Y    \\
      =&\int_0^\infty\int_{|\fai|=1} r^{-1-\beta}e^{-\lambda r}(\sin(r\k\cdot\fai)-r\k\cdot\fai)m(\fai)d\fai dr \\
      =&\frac{\lambda^2}{(-\beta)(1-\beta)}\int_0^\infty r^{1-\beta}e^{-\lambda r} \int_{|\fai|=1}(\sin(r\k\cdot\fai)-r\k\cdot\fai)m(\fai)d\fai dr \\
      &-\frac{2\lambda}{(-\beta)(1-\beta)}\int_0^\infty r^{1-\beta}e^{-\lambda r} \int_{|\fai|=1}(\k\cdot\fai) (\cos(r\k\cdot\fai)-1) m(\fai)d\fai dr \\
      &-\frac{1}{(-\beta)(1-\beta)}\int_0^\infty r^{1-\beta}e^{-\lambda r} \int_{|\fai|=1} (\k\cdot\fai)^2 \sin(r\k\cdot\fai) m(\fai)d\fai dr \\
      =& -\Gamma(-\beta) \int_{|\fai|=1} (\lambda^2+(\k\cdot\fai)^2)^\frac{\beta}{2} \sin(\beta\eta) ~m(\fai)d\fai \\
       & - \Gamma(1-\beta)\lambda^{\beta-1} \int_{|\fai|=1} (\k\cdot\fai)m(\fai)d\fai .
    \end{split}
  \end{equation}
  Combining \eqref{equiv1} and \eqref{equiv2} leads to the  anisotropic tempered fractional Laplacian in Fourier space
  \begin{equation*}
  \begin{split}
  \mathscr{F}[\Delta_m^{\beta/2,\lambda}p(\X,t)]
  &= (-1)^{\lceil\beta\rceil} \int_{|\fai|=1} \Big( (\lambda^2+(\k\cdot\fai)^2)^{\frac{\beta}{2}} e^{-i\beta\eta} - \lambda^\beta \Big)~m(\fai)d\fai\cdot \hat{p}(\k,t), \\
  &= (-1)^{\lceil\beta\rceil} \int_{|\fai|=1} \Big( (\lambda- i\k\cdot\fai)^\beta  - \lambda^\beta \Big)~m(\fai)d\fai\cdot \hat{p}(\k,t),
  \end{split}
  \end{equation*}
  which equals to \eqref{anisoLapTempFT}.
\end{proof}

%=========================================================================================
%------------------------------------ Bibliography ---------------------------------------
%\addcontentsline{toc}{section}{References}

\end{document}